\documentclass[12pt]{article}
\usepackage{amsthm,amsfonts,amssymb,amsmath}
\usepackage{mathtools}
\usepackage{cite}
\usepackage{epsfig}
\usepackage{url}
\usepackage{xcolor,tikz}
\usetikzlibrary{matrix}
\usepackage{multicol}
\usepackage{authblk,fullpage,breqn,hyperref}
\usepackage{stmaryrd}
\usepackage{caption}
\usepackage{subcaption}

\usepackage[shortlabels]{enumitem}

\newcommand{\newConj}[2]{
  \expandafter\newcommand\csname conjText#1\endcsname{#2}
\begin{conj}
\label{conj:#1}
\csname conjText#1\endcsname
\end{conj}
}

\newcommand{\newQues}[2]{
  \expandafter\newcommand\csname quesText#1\endcsname{#2}
\begin{ques}
\label{ques:#1}
\csname quesText#1\endcsname
\end{ques}
}

\newcommand{\newProb}[2]{
  \expandafter\newcommand\csname probText#1\endcsname{#2}
\begin{prob}
\label{prob:#1}
\csname probText#1\endcsname
\end{prob}
}

\newcommand{\newThm}[2]{
  \expandafter\newcommand\csname thmText#1\endcsname{#2}
\begin{thm}
\label{th:#1}
\csname thmText#1\endcsname
\end{thm}
}

\newcommand{\newLem}[2]{
  \expandafter\newcommand\csname lemText#1\endcsname{#2}
\begin{lem}
\label{lem:#1}
\csname lemText#1\endcsname
\end{lem}
}

\newcommand{\repeatConj}[1]{
\csname theoremstyle\endcsname{plain}
\csname newtheorem\endcsname*{#1ConjRepeat}{Conjecture~\csname ref\endcsname{conj:#1}}
\csname begin\endcsname{#1ConjRepeat}
\csname conjText#1\endcsname
\csname end\endcsname{#1ConjRepeat}
}

\newcommand{\repeatQues}[1]{
\csname theoremstyle\endcsname{definition}
\csname newtheorem\endcsname*{#1QuesRepeat}{Question~\csname ref\endcsname{ques:#1}}
\csname begin\endcsname{#1QuesRepeat}
\csname quesText#1\endcsname
\csname end\endcsname{#1QuesRepeat}
}

\newcommand{\repeatProb}[1]{
\csname theoremstyle\endcsname{definition}
\csname newtheorem\endcsname*{#1ProbRepeat}{Problem~\csname ref\endcsname{prob:#1}}
\csname begin\endcsname{#1ProbRepeat}
\csname probText#1\endcsname
\csname end\endcsname{#1ProbRepeat}
}

\newcommand{\repeatThm}[1]{
\csname theoremstyle\endcsname{plain}
\csname newtheorem\endcsname*{#1ThmRepeat}{Theorem~\csname ref\endcsname{th:#1}}
\csname begin\endcsname{#1ThmRepeat}
\csname thmText#1\endcsname
\csname end\endcsname{#1ThmRepeat}
}

\newcommand{\repeatLem}[1]{
\csname theoremstyle\endcsname{plain}
\csname newtheorem\endcsname*{#1LemRepeat}{Lemma~\csname ref\endcsname{lem:#1}}
\csname begin\endcsname{#1LemRepeat}
\csname lemText#1\endcsname
\csname end\endcsname{#1LemRepeat}
}

\numberwithin{equation}{section}

\setlist{leftmargin=3\parindent,labelindent=3\parindent}
\setlist[enumerate]{%
  leftmargin=3\parindent,%
  align=left,%
  labelwidth=3\parindent,%
  labelsep=0pt%
}
\setlist[enumerate,1]{%
  label={\normalfont (\thesection.\arabic{equation})}, ref={\normalfont \thesection.\arabic{equation}},
  resume%
}

\makeatletter
\newtheorem*{rep@theorem}{\rep@title}
\newcommand{\newreptheorem}[2]{%
\newenvironment{rep#1}[1]{%
 \def\rep@title{#2 \ref{##1}}%
 \begin{rep@theorem}}%
 {\end{rep@theorem}}}
\makeatother

\newtheorem{thm}[equation]{Theorem}
\newtheorem{cor}[equation]{Corollary}
\newtheorem{lem}[equation]{Lemma}
\newtheorem{prop}[equation]{Proposition}
\newtheorem{claim}[equation]{Claim}

\newtheorem{conj}[equation]{Conjecture}
\newreptheorem{lemma}{Lemma}

\theoremstyle{definition}
\newtheorem{defn}[equation]{Definition}
\newtheorem{ex}[equation]{Example}

\newtheorem{rem}[equation]{Remark}

\newtheorem{ques}[equation]{Question}
\newtheorem{prob}[equation]{Problem}

\theoremstyle{remark}

\title{Off-Diagonal Commonality of Graphs via Entropy}
\author{Natalie Behague\thanks{Supported by a PIMS Postdoctoral Fellowship.} }
\author{Natasha  Morrison\thanks{Research supported by NSERC Discovery Grant RGPIN-2021-02511, NSERC Early Career Supplement DGECR-2021-00047 and a Start-Up Grant from the University of Victoria.}}
\author{Jonathan A. Noel\thanks{Research supported by NSERC Discovery Grant RGPIN-2021-02460, NSERC Early Career Supplement DGECR-2021-00024 and a Start-Up Grant from the University of Victoria.}} 

\affil{\normalsize{Department of Mathematics and Statistics, University of Victoria, Victoria, B.C., Canada.}}
\affil{\texttt{\{nbehague,nmorrison,noelj\}@uvic.ca}}

\DeclareMathOperator{\Hom}{Hom}
\DeclareMathOperator{\Aut}{Aut}

\DeclareMathOperator{\rng}{rng}

\DeclareTextCompositeCommand{\v}{OT1}{l}{l\nobreak\hspace{-.1em}'}
\DeclareTextCompositeCommand{\v}{OT1}{t}{t\nobreak\hspace{-.1em}'\nobreak\hspace{-.15em}}

\begin{document}

\maketitle

\begin{abstract}
A graph $H$ is \emph{common} if the limit as $n\to\infty$ of the minimum density of monochromatic labelled copies of $H$ in an edge colouring of $K_n$ with red and blue is attained by a sequence of quasirandom colourings. We apply an information-theoretic approach to show that certain graphs obtained from odd cycles and paths via gluing operations are common. In fact, for every pair $(H_1,H_2)$ of such graphs, there exists $p\in(0,1)$ such that an appropriate linear combination of red copies of $H_1$ and blue copies of $H_2$ is minimized by a quasirandom colouring in which $p\binom{n}{2}$ edges are red; such a pair $(H_1,H_2)$ is said to be \emph{$(p,1-p)$-common}. Our approach exploits a strengthening of the common graph property for odd cycles that was recently proved using Schur convexity. We also exhibit a $(p,1-p)$-common pair $(H_1,H_2)$ such that $H_2$ is uncommon. 
\end{abstract}

\section{Introduction}

A \emph{homomorphism} from a graph $H$ to a graph $G$ is a function $f:V(H)\to V(G)$ such that $f(u)f(v)\in E(G)$ whenever $uv\in E(H)$. The \emph{homomorphism density} of $H$ in $G$, denoted $t(H,G)$, is the probability that a uniformly random function from $V(H)$ to $V(G)$ is a homomorphism. The homomorphism density function has played a central role in the development of the celebrated theory of graph limits~\cite{Lovasz12} and is ubiquitous in modern extremal graph theory. For instance, the famous Sidorenko Conjecture~\cite{Sidorenko93} states that, if $H$ is a non-empty\footnote{A graph is \emph{non-empty} if its edge set is non-empty.} bipartite graph, then
\begin{equation}\label{eq:Sid}t(H,G)\geq t(K_2,G)^{e(H)}\end{equation}
for every graph $G$, where $e(H):=|E(H)|$ and $v(H):=|V(H)|$. A graph $H$ satisfying this conjecture is said to be \emph{Sidorenko}. In other words, a graph $H$ is Sidorenko if the homomorphism density of $H$ among all graphs $G$ of a given edge density is asymptotically minimized by taking $G$ to be a binomial random graph in which each edge appears with probability $t(K_2,G)$. A well-studied notion, with close ties to Sidorenko's Conjecture, is that of a common graph; a non-empty graph $H$ is said to be \emph{common} if
\begin{equation}\label{eq:common}t(H,G) + t(H,\overline{G})\geq (1/2)^{e(H)-1}-o(1)\end{equation}
for every graph $G$, where $\overline{G}$ is the complement of $G$ and the $o(1)$ term tends to zero as $v(G)\to\infty$. By thinking of the edges of $G$ and $\overline{G}$ as being red or blue, respectively, one gets that $H$ is common if the number of (homomorphic) copies of $H$ in a red/blue colouring of the edges of a complete graph is asymptotically minimized by an uniformly random colouring. 

Our focus in this paper is on leveraging inequalities related to \eqref{eq:Sid} and \eqref{eq:common} to obtain new examples of common graphs and pairs $(H_1,H_2)$ of graphs which satisfy an ``off-diagonal'' generalization of the common graph property (see Definition~\ref{defn:commonPair}). To describe our results, it is convenient to work in the language of graph limits. A \emph{kernel} is a bounded measurable function $U:[0,1]^2\to \mathbb{R}$ such that $U(x,y)=U(y,x)$ for all $x,y\in [0,1]$, and a \emph{graphon} is a kernel $W$ such that $0\leq W\leq1$. The \emph{homomorphism density} of a graph $H$ in a kernel $U$ is
\[t(H,U):=\int_{[0,1]^{V(H)}}\prod_{uv\in E(H)}U(x_u,x_v)\prod_{v\in V(H)}dx_v.\]
The set of graphons can be thought of as the ``completion'' of the set of dense graphs~\cite{LovaszSzegedy06}. In particular, every asymptotic inequality involving homomorphism density function in graphs is equivalent to an analogous inequality in graphons. E.g. a non-empty graph $H$ is Sidorenko if and only if $t(H,W)\geq t(K_2,W)^{e(H)}$ for every graphon $W$ and common if and only if $t(H,W)+t(H,1-W)\geq (1/2)^{e(H)-1}$ for every graphon $W$; for a more general statement, see Lemma~\ref{lem:limit}. The following definition is key to the approach of this paper.

\begin{defn}
Say that a non-empty graph $H$ is \emph{strongly common} if
\[t(H,W)+t(H,1-W)\geq t(K_2,W)^{e(H)}+t(K_2,1-W)^{e(H)}\]
for every graphon $W$. 
\end{defn}

Clearly, if $H$ is Sidorenko, then it is strongly common and, if $H$ is strongly common, then it is common. A classical result of Goodman~\cite{Goodman59} (see Theorem~\ref{th:Goodman}) implies that $K_3$ is strongly common. Our first result is that the same is true for the $5$-cycle.

\newThm{C5}{$C_5$ is strongly common.}

In a preprint of~\cite{FirstPaper}, we made the following conjecture.\footnote{After submitting~\cite{FirstPaper} to arxiv, we decided to split it into two papers, one of which is the present paper. Therefore, most of the results, open problems, etc, in this paper were contained in the early preprints of~\cite{FirstPaper}.}

\begin{conj}
\label{conj:oddCycles}
All odd cycles are strongly common.
\end{conj}

Kim and Lee~\cite{KimLee22+} recently confirmed Conjecture~\ref{conj:oddCycles} using a novel approach based on Schur convexity (see Theorem~\ref{th:OddCycles}). Chen and Ma~\cite{ChenMa23+} proved that the only strongly common graph containing a triangle is $K_3$ itself. This was very recently extended by Versteegen~\cite{Versteegen23+} to the statement that the only strongly common graphs of odd girth are the odd cycles.  Lee and Noel~\cite{LeeNoel23+} recently applied the fact that $K_3$ is strongly common to find the first example of an uncommon graph $H$ such that the disjoint union $H\sqcup H$ is common. 

In this paper, we use the fact that odd cycles are strongly common to obtain new examples of common graphs. These graphs are built up from odd cycles and paths via specific gluing operations. As a special case, say that a graph $H$ is a \emph{simple $C_m$-tree} if either $H=C_m$ or $H$ can be obtained from a smaller simple $C_m$-tree, say $H'$, by adding a copy of $C_m$ and identifying one vertex or edge of the new copy of $C_m$ with a vertex or edge, respectively, of $H'$. Say that $H$ is a \emph{$C_m$-vertex tree} or a \emph{$C_m$-edge tree} if every identification in the construction of $H$ involves a vertex or edge, respectively. Sidorenko~\cite{Sidorenko96} proved that all $K_3$-vertex trees and $K_3$-edge trees are common. These results were unified by Grzesik, Lee, Lidick\'y and Volec~\cite{GrzesikLeeLidickyVolec22} to the statement that all simple $K_3$-trees are common; as discussed in~\cite[Section~5]{GrzesikLeeLidickyVolec22}, their proof easily extends to show that simple $C_m$-trees are common for all odd $m\geq3$. 

We generalize this result of~\cite{GrzesikLeeLidickyVolec22} in two ways. First, we partially solve an open problem posed in~\cite[Section~5]{GrzesikLeeLidickyVolec22}, which asks whether one can obtain common graphs from odd cycles via more general gluing operations than those used to construct simple $C_m$-trees. Our most general result (Theorem~\ref{th:cycleTree}) applies to graphs obtained from gluing together subgraphs of odd cycles under certain technical conditions. The next result provides three specific examples which illustrate the types of gluing operations that we can handle.

\newThm{Sample}{The three graphs depicted in Figure~\ref{fig:glueExamples} are common. }

\begin{figure}[htbp]
    \centering
         \centering
         \includegraphics[scale=1]{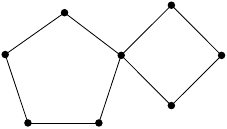}\quad
         \includegraphics[scale=1]{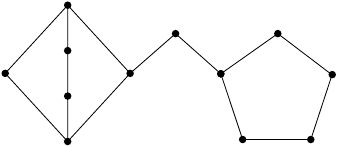}\quad
         \includegraphics[scale=1]{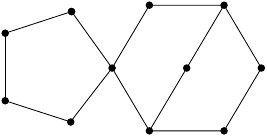}
    \caption{Three new examples of common graphs.}
    \label{fig:glueExamples}
\end{figure}

We also extend the results of~\cite{GrzesikLeeLidickyVolec22} to the following off-diagonal generalization of the notion of common graphs, which was recently introduced in~\cite{FirstPaper}. 

\begin{defn}
\label{defn:commonPair}
For $p_1,p_2\in (0,1)$ such that $p_1+p_2=1$, a pair $(H_1,H_2)$ of non-empty graphs is said to be \emph{$(p_1,p_2)$-common} if the following holds for any graphons $W_1$ and $W_2$ such that $W_1+W_2=1$:
\begin{equation}
\label{eq:commonPair}
\frac{t(H_1,W_1)}{e(H_1)p_1^{e(H_1)-1}} + \frac{t(H_2,W_2)}{e(H_2)p_2^{e(H_2)-1}}\geq \frac{p_1}{e(H_1)}+\frac{p_2}{e(H_2)}.
\end{equation}
\end{defn}

It is easily observed that a graph $H$ is common if and only if $(H,H)$ is $(1/2,1/2)$-common, and so Definition~\ref{defn:commonPair} extends the notion of a common graph to an ``off-diagonal'' setting. As another application of our most general result, Theorem~\ref{th:cycleTree}, we characterize pairs $(p_1,p_2)$ for which $(H_1,H_2)$ is $(p_1,p_2)$-common when $H_1$ and $H_2$ are simple $C_m$-trees for odd $m$.

\newThm{SimpleTrees}{Let $m\geq3$ be odd, let $H_1$ and $H_2$ be simple $C_m$-trees and let $p_1,p_2\in (0,1)$ such that $p_1+p_2=1$. Then $(H_1,H_2)$ is $(p_1,p_2)$-common if and only if
\[\frac{e(H_1)-v(H_1)+1}{e(H_1)p_1^{m-1}}=\frac{e(H_2)-v(H_2)+1}{e(H_2)p_2^{m-1}}.\]}

The study of common graphs was inspired by the 1959 result of Goodman~\cite{Goodman59} which implies that $K_3$ is common. Shortly after this, Erd\H{o}s~\cite{Erdos62} conjectured that all complete graphs are common. More boldly, Burr and Rosta~\cite{BurrRosta80} conjectured that every non-empty graph is common. The latter conjecture was disproved by Sidorenko~\cite{Sidorenko89}, who showed that the triangle with a pendant edge (known as the \emph{paw} graph) is uncommon; the same construction shows that many other graphs, e.g., $K_3\sqcup K_2$, are uncommon. Around the same time, Thomason~\cite{Thomason89} showed that $K_k$ is uncommon for all $k\geq4$, thereby refuting the conjecture of Erd\H{o}s~\cite{Erdos62}. Jagger, \v{S}\v{t}ov\'i\v{c}ek and Thomason~\cite{JaggerStovicekThomason96} generalized Thomason's~\cite{Thomason89} theorem to the statement that every common graph is $K_4$-free. 

In an early arxiv version of~\cite{FirstPaper}, we conjectured that there exists a pair $(H_1,H_2)$ and $p\in (0,1)$ such that $H_2$ is uncommon and $(H_1,H_2)$ is $(p,1-p)$-common. Since Sidorenko's~\cite{Sidorenko89} construction proves that $K_3\sqcup K_2$ is uncommon, our next result resolves this conjecture in the affirmative. Let $D$ denote the graph obtained from $K_4$ by deleting one edge, which is referred to as the \emph{diamond} graph. 

\newThm{K3unionK2}{$(D,K_3\sqcup K_2)$ is $(p,1-p)$-common for $p=\frac{8-2\sqrt{10}}{3}$.}

For context, we remark that it is not true that every graph $H$ is contained in a $(p,1-p)$-common pair for some $p\in (0,1)$. Indeed, one of the results of~\cite{FirstPaper} says that, if $H$ contains a $K_4$, then $H$ is not contained in a $(p,1-p)$-common pair for any $p\in (0,1)$. Therefore, in finding a $(p,1-p)$-common pair $(H_1,H_2)$ such that $H_2$ is uncommon, one needs to choose the graph $H_2$ somewhat carefully. 

Various other problems in the area of common graphs have been well studied. One of the oldest and most well-known problems has been to determine whether there exist common graphs of arbitrary chromatic number~\cite{Hatami+12,JaggerStovicekThomason96}. The first example of a common graph of chromatic number four was obtained by Hatami, Hladk\'y, Kr\'a\v{l}, Norine and Razborov~\cite{Hatami+12} using the powerful flag algebra method. More recently, Kr\'a\v{l}, Volec and Wei~\cite{KralVolecWei22+} settled this problem in its entirety by finding connected common graphs of every chromatic number. Ko and Lee~\cite{KoLee22} strengthened this to the statement that there are common graphs of every chromatic number with high connectivity. Multicolour extensions of the notion of common graphs have also been studied in, e.g.,~\cite{FirstPaper,Kral+22,Cummings+13,JaggerStovicekThomason96}. 

This paper is structured as follows.
In Section~\ref{sec:strongly}, we combine a standard ``algebraic expansion technique'' for kernels with recent results on homomorphism densities of paths to prove Theorem~\ref{th:C5}. Then, in Section~\ref{sec:entropy}, we recall some basic facts about the entropy of a random variable and state a ``convexity lemma'' which allows us to use examples of strongly common graphs and binomial inequalities for homomorphism densities to obtain new examples of $(p_1,p_2)$-common pairs of graphs. In Section~\ref{sec:example}, we prove Theorem~\ref{th:Sample} for the first graph in Figure~\ref{fig:glueExamples}; this serves as a relatively simple example to illustrate the applicability of the tools introduced in Section~\ref{sec:entropy}. We then generalize this approach in Section~\ref{sec:gluing} to obtain a binomial inequality involving homomorphism densities of certain graphs built up from subgraphs of a fixed graph $F$ via gluing operations. In Section~\ref{sec:oddCycles}, we feed these binomial inequalities in the case $F=C_m$ for odd $m\geq3$ into the convexity lemma from Section~\ref{sec:entropy} to obtain a result (Theorem~\ref{th:cycleTree}) which implies both of Theorems~\ref{th:Sample} and~\ref{th:SimpleTrees}. In Section~\ref{sec:supersat}, we prove Theorem~\ref{th:K3unionK2}; the proof involves reducing the problem of showing that $(D,K_3\sqcup K_2)$ is $(p,1-p)$-common to a certain optimization problem, where the constraints are implied by a classical supersaturation theorem in extremal graph theory. We conclude the paper in Section~\ref{sec:concl} by proposing some open problems. The proof of the convexity lemma from Section~\ref{sec:entropy} is provided in Appendix~\ref{app:convexity}.

\section{Strongly Common Graphs}
\label{sec:strongly}

Our goal in this section is to prove Theorem~\ref{th:C5}. As a warm-up, we prove the analogous statement for $K_3$, which is known as Goodman's Formula~\cite{Goodman59}. The proofs in this section make use of the following standard lemma. Given a graph $H$ and a set $E \subseteq E(H)$, let $H[E]$ be the graph with vertex set $V(H)$ and edge set $E$.

\begin{lem}
\label{lem:expansion}
For every graph $H$, kernel $W$ and $p\in \mathbb{R}$, if $U=W-p$, then
\[t(H,W)=\sum_{E\subseteq E(H)}p^{e(H)-|E|}t(H[E],U).\]
\end{lem}

\begin{proof}
We have
\[t(H,W)=\int_{[0,1]^{V(H)}}\prod_{uv\in E(H)}W(x_u,x_v)\prod_{v\in V(H)}dx_v = \int_{[0,1]^{V(H)}}\prod_{uv\in E(H)}(U(x_u,x_v)+p)\prod_{v\in V(H)}dx_v.\]
The result follows by expanding the product in the integrand. 
\end{proof}

For $k\geq1$, let $P_k$ denote the path on $k$ vertices.

\begin{thm}[Goodman's Formula~\cite{Goodman59}]
\label{th:Goodman}
If $W_1$ and $W_2$ are kernels satisfying $W_1+W_2=1$, then
\[t(K_3,W_1)+t(K_3,W_2) = \sum_{i=1}^2\left[t(K_2,W_i)^3 +\frac{3}{2}(t(P_3,W_i) - t(K_2,W_i)^2)\right].\]
\end{thm}

\begin{proof}
Let $W_1$ and $W_2$ be kernels such that $W_1+W_2=1$ and let $U$ be the kernel defined by $U:= W_1-t(K_2,W_1)$. Note that, since $W_1+W_2=1$, we have $W_2=t(K_2,W_2)-U$. Also,
\begin{equation}\label{eq:Uzero}t(K_2,U) = \int_0^1\int_0^1(W_1(x_1,x_2)-t(K_2,W_1))dx_1dx_2 = 0.\end{equation}
Using Lemma~\ref{lem:expansion}, we get that 
\[t(K_3,W_1)=t(K_2,W_1)^3 + 3t(K_2,W_1)^2t(K_2,U)+3t(K_2,W_1)t(P_3,U) + t(K_3,U)\]
and 
\[t(K_3,W_2)=t(K_2,W_2)^3 - 3t(K_2,W_2)^2t(K_2,U)+3t(K_2,W_2)t(P_3,U) - t(K_3,U).\]
Summing these two quantities and applying \eqref{eq:Uzero} yields
\[t(K_3,W_1)+t(K_3,W_2)=t(K_2,W_1)^3 + t(K_2,W_2)^3 + 3(t(K_2,W_1)+t(K_2,W_2))t(P_3,U)\]
\[=t(K_2,W_1)^3 + t(K_2,W_2)^3 + 3t(P_3,U).\]
Now, applying Lemma~\ref{lem:expansion} to $t(P_3,W_1)+t(P_3,W_2)$, and using \eqref{eq:Uzero} again, we get that 
\[t(P_3,W_1)+t(P_3,W_2) = t(K_2,W_1)^2+t(K_2,W_2)^2 +2t(P_3,U).\]
Solving for $t(P_3,U)$ in this expression and substituting it into the expression for $t(K_3,W_1)+t(K_3,W_2)$ derived above completes the proof. 
\end{proof}

\begin{thm}[Goodman~\cite{Goodman59}]
\label{th:K3stronglycommon}
$K_3$ is strongly common.
\end{thm}

\begin{proof}
The result follows immediately from Theorem~\ref{th:Goodman} and the well-known fact that $P_3$ is Sidorenko. 
\end{proof}

The next theorem provides an analogue of Theorem~\ref{th:Goodman} for the $5$-cycle, from which Theorem~\ref{th:C5} will be derived. 

\begin{thm}
\label{th:C5Goodman}
If $W_1$ and $W_2$ are kernels satisfying $W_1+W_2=1$, then
\[t(C_5,W_1) +t(C_5,W_2)=\sum_{i=1}^2\left[t(K_2,W_i)^5+ 5t(K_2,W_i)t(P_5,W_i)- 5t(K_2,W_i)^2t(P_4,W_i)\right].\]
\end{thm}

\begin{proof}
Let $q_1:=t(K_2,W_1)$ and $q_2:=t(K_2,W_1)$ and define $U:=q_1-W_1$. As in the proof of Theorem~\ref{th:Goodman}, we have $t(K_2,U)=0$. By Lemma~\ref{lem:expansion}, we get that $t(C_5,W_1)+t(C_5,W_2)$ is
\[
q_1^5 + q_2^5 + 5(q_1^3+q_2^3)t(P_3,U) + 5(q_1^2-q_2^2)t(P_4,U) + 5(q_1+q_2)t(P_5,U).
\]
Applying Lemma~\ref{lem:expansion} to the sum
 $t(K_2,W_1)t(P_5,W_1)+t(K_2,W_2)t(P_5,W_2)$, we get that it equals
\[q_1( q_1^4 + 3q_1^2t(P_3,U) + 2q_1t(P_4,U) + t(P_5,U)) + q_2( q_2^4 + 3q_2^2t(P_3,U) - 2q_2t(P_4,U) + t(P_5,U)) \]
\[ = q_1^5 + q_2^5  + 3(q_1^3 + q_2^3)t(P_3,U) + 2(q_1^2-q_2^2)t(P_4,U)
 + (q_1+q_2)t(P_5,U).\]
Similarly, $t(K_2,W_1)^2t(P_4,W_1) - t(K_2,W_1)^2t(P_4,W_1)$ is equal to 
\[ q_1^2( q_1^3 + 2q_1t(P_3,U) + t(P_4,U)) + q_2^2( q_2^3 + 2q_2t(P_3,U) - t(P_4,U)) \]
\[ = q_1^5 + q_2^5  + 2(q_1^3 + q_2^3)t(P_3,U) + (q_1^2-q_2^2)t(P_4,U).\]
The result follows by substituting these quantities into the expression for $t(C_5,W_1)+t(C_5,W_2)$ derived above and simplifying.
\end{proof}

By Theorem~\ref{th:C5Goodman}, the problem of showing that $C_5$ is strongly common reduces to establishing an inequality for homomorphism densities of paths. For this, we apply the following recent result of Blekherman and Raymond~\cite{BlekhermanRaymond22}.

\begin{lem}[Blekherman and Raymond~{\cite[Theorem~1.3~(1.1)]{BlekhermanRaymond22}}]
\label{lem:Blekherman}
Let $0 \le r \le s \le t$ be integers such that $r$ and $t$ are odd. Then, for every graphon $W$,
\[t(P_r,W)^{t-s} t(P_t,W)^{s-r} \ge t(P_s,W)^{t-r}.\]
\end{lem}

\begin{cor} \label{cor:P6vsP3P4}
For every graphon $W$, 
\[t(P_5,W)\geq t(K_2,W)t(P_4,W).\]
\end{cor}

\begin{proof}
Let $W$ be a graphon. By Lemma~\ref{lem:Blekherman},
\begin{align*}
t(P_1,W)^{3} t(P_5,W) &\ge t(P_2,W)^{4} \quad \text{and}\\
t(P_1,W) t(P_5,W)^{3} &\ge t(P_4,W)^{4}.
\end{align*}
The result follows by multiplying these two inequalities and taking the fourth root. 
\end{proof}

We now prove Theorem~\ref{th:C5}, which we restate here for convenience. 

\repeatThm{C5}

\begin{proof}
This is an immediate consequence of Theorem~\ref{th:C5Goodman} and Corollary~\ref{cor:P6vsP3P4}.
\end{proof}

As mentioned in the introduction,  Kim and Lee~\cite{KimLee22+} recently proved that all odd cycles are strongly common. In fact, they obtained the following stronger result which applies to all kernels, not just graphons. We apply this result in Section~\ref{sec:oddCycles}.

\begin{thm}[Kim and Lee~\cite{KimLee22+}]
\label{th:OddCycles}
For every kernel $W$ and odd $m\geq3$,
\[t(C_m,W) + t(C_m,1-W)\geq t(K_2,W)^m+t(K_2,1-W)^m.\]
In particular, all odd cycles are strongly common.
\end{thm}

\section{Entropy and Convexity}
\label{sec:entropy}

The following lemma is is useful in applying Theorem~\ref{th:OddCycles} to obtain examples of $(p_1,p_2)$-common pairs of graphs. The proof is a somewhat technical calculus exercise which we have chosen to put into the appendix. 

\newLem{convexity}{
Let $F$ be strongly common and let $H_1$ and $H_2$ be non-empty graphs. If $k_1,k_2,\ell_1,\ell_2$ are non-negative integers, $p_1,p_2\in (0,1)$ such that $p_1+p_2=1$ and conditions \eqref{eq:H_ge_F}--\eqref{eq:thecorrelation} below are satisfied, then $(H_1,H_2)$ is $(p_1,p_2)$-common.
\begin{enumerate}
\stepcounter{equation}
    \item\label{eq:H_ge_F} $e(H_i) \ge e(F)$  for $i\in\{1,2\}$,
\stepcounter{equation}
    \item \label{eq:ehef} $e(H_i)=k_ie(F)-\ell_i$ for $i\in \{1,2\}$,
\stepcounter{equation}
    \item\label{eq:p1p2} $\frac{k_1}{e(H_1)p_1^{e(F)-1}}= \frac{k_2}{e(H_2)p_2^{e(F)-1}}$,
\stepcounter{equation}
    \item\label{eq:thecorrelation} $t(H_i,W)t(K_2,W)^{\ell_i}\geq t(F,W)^{k_i}$ for every graphon $W$ and $i\in\{1,2\}$.
\end{enumerate}
}

Applying Lemma~\ref{lem:convexity} to the case $H_1=H_2$ yields the following corollary.

\begin{cor}
\label{cor:convexity}
If $F$ is strongly common, $H$ is a non-empty graph and $k,\ell$ are non-negative integers such that conditions \eqref{eq:H_ge_Fcor}--\eqref{eq:thecorrelationcor} below are satisfied, then $H$ is common.
\begin{enumerate}
\stepcounter{equation}
    \item\label{eq:H_ge_Fcor} $e(H) \ge e(F)$,
\stepcounter{equation}
    \item \label{eq:ehefcor} $e(H)=ke(F)-\ell$,
\stepcounter{equation}
    \item\label{eq:thecorrelationcor} $t(H,W)t(K_2,W)^{\ell}\geq t(F,W)^{k}$ for every graphon $W$.
\end{enumerate}
\end{cor}

In the rest of this section, we review some basic properties of the entropy of a random variable that will be used in the next two sections. Given a discrete random variable $X$, the \emph{range} of $X$, denoted $\rng(X)$, is the set of all $x$ such that $\mathbb{P}(X=x)>0$.

\begin{defn}
Let $X$ be a discrete random variable. The \emph{entropy} of $X$ is
\[\mathbb{H}(X) := \sum_{x\in \rng(X)}\mathbb{P}(X=x)\log_2\left(\frac{1}{\mathbb{P}(X=x)}\right).\] 
\end{defn}

The next lemma follows by applying Jensen's Inquality to the convex function $x\log_2(x)$.

\begin{lem}
\label{lem:maxEntropy}
If $X$ is a discrete random variable, then
\[\mathbb{H}(X) \leq \log_2(|\rng(X)|).\]
Moreover, $\mathbb{H}(X)=\log_2(|\rng(X)|)$ if and only if  $X$ is uniformly distributed on its range. 
\end{lem}

For discrete random variables $X$ and $Y$ and $y\in\rng(Y)$, let $\rng(X\mid Y=y)$ be the set of all $x$ such that $\mathbb{P}(X=x\mid Y=y)>0$. 

\begin{defn}
Let $X$ and $Y$ be discrete random variables. For $y\in \rng(Y)$, define 
\[\mathbb{H}(X\mid Y=y):= \sum_{x\in \rng(X\mid Y=y)}\mathbb{P}(X=x\mid Y=y)\log_2\left(\frac{1}{\mathbb{P}(X=x\mid Y=y)}\right).\] 
\end{defn}

\begin{defn}
\label{defn:conditionalEntropy}
For two discrete random variables $X$ and $Y$, the \emph{conditional entropy of $X$ given $Y$} is defined by
\[\mathbb{H}(X\mid Y) := \sum_{y\in \rng(Y)}\mathbb{P}(Y=y)\mathbb{H}(X\mid Y=y).\]
\end{defn}

The next lemma, known as the ``chain rule'' for entropy, is very useful for analyzing the entropy of a tuple of random variables.

\begin{lem}[Chain Rule]
\label{lem:chainRule}
For $m\geq1$ and $k\geq0$ and discrete random variables  $X_1,X_2,\dots,X_m$ and $Y_1,Y_2,\dots,Y_k$, we have
\[\mathbb{H}(X_1,\dots,X_m\mid Y_1,\dots,Y_k) = \sum_{i=1}^m\mathbb{H}(X_i\mid Y_1,\dots,Y_k,X_1,\dots,X_{i-1}).\]
\end{lem}

Another standard fact is that conditioning on a larger set of random variables can only decrease conditional entropy. 

\begin{defn}
For discrete random variables $X,Y$ and $Z$, say that $X$ and $Y$ are \emph{conditionally independent given $Z$} if the following hold for every $z\in \rng(Z)$, $x\in \rng(X\mid Z=z)$ and $y\in \rng(Y\mid Z=z)$:
\[\mathbb{P}(X=x\mid Z=z,Y=y)=\mathbb{P}(X=x\mid Z=z)\text{ and}\]
\[\mathbb{P}(Y=y\mid Z=z, X=x)=\mathbb{P}(Y=y\mid Z=z).\]
\end{defn}

\begin{lem}[Deconditioning]
\label{lem:deconditioning}
For $m,k\geq1$ and discrete random variables  $X_1,X_2,\dots,X_m$ and $Y_1,Y_2,\dots,Y_k$, we have
\[\mathbb{H}(X_1,\dots,X_m\mid Y_1,\dots, Y_k)\leq \mathbb{H}(X_1,\dots,X_m\mid Y_1,\dots, Y_{k-1}).\]
Moreover, equality holds if and only if $(X_1,\dots,X_m)$ and $Y_k$ are conditionally independent given $(Y_1,\dots,Y_{k-1})$. 
\end{lem}

We conclude this section with a lemma which is useful for constructing high entropy distributions with specific marginal distributions. The ``consequently'' part of this lemma uses Lemma~\ref{lem:chainRule} and the ``moreover'' part of Lemma~\ref{lem:deconditioning}. 

\begin{lem}[See~{\cite[Lemma~2.5]{Lee21}}]
\label{lem:glueEntropy}
Let $A_1,A_2,A_2'$ and $A_3$ be discrete random variables. If $A_2$ and $A_2'$ are identically distributed, then there exist $B_1, B_2$ and $B_3$ such that $B_1$ and $B_3$ are conditionally independent given $B_2$, $(B_1,B_2)$ and $(A_1,A_2)$ are identically distributed and $(B_2,B_3)$ and $(A_2',A_3)$ are identically distributed. Consequently,
\begin{align*}\mathbb{H}(B_1,B_2,B_3)&= \mathbb{H}(A_1,A_2)+\mathbb{H}(A_3\mid A_2')\\
&= \mathbb{H}(A_1,A_2)+\mathbb{H}(A_2',A_3) - \mathbb{H}(A_2).\end{align*}
\end{lem}

\section{A Worked Example}
\label{sec:example}

The goal of this section is to use the ideas built up in the previous section to show that the first graph depicted in Figure~\ref{fig:glueExamples} is common. This will provide a specific example to illustrate the approach before diving into the details of the (more abstract) general case.

For two graphs $J$ and $G$, let $\Hom(J,G)$ be the set of homomorphisms from $J$ to $G$ and $\hom(J,G):=|\Hom(J,G)|$. We say that a random variable $(X_v: v\in V(J))$ is \emph{$G$-homomorphism supported} if, for almost every choice of $(X_v: v\in V(J))$, there exists $f\in \Hom(J,G)$  such that $X_v=f(v)$ for all $v\in V(J)$.  We start by establishing the following proposition.

\begin{prop}
\label{prop:H1}
Let $H$ be the first graph depicted in Figure~\ref{fig:glueExamples}. The graph $J=H\sqcup K_2$ satisfies $\hom(J,G)\geq \hom(C_5,G)^2$ for every graph $G$.
\end{prop}

\begin{proof}
Denote the vertices of $C_5$ by $1,2,\dots,5$ in cyclic order. If $\hom(C_5,G)=0$, then the inequality holds trivially. So, let $G$ be a graph with $\hom(C_5,G)>0$. Let $f$ be a uniformly random homomorphism from $C_5$ to $G$ and let $X_i=f(i)$ for $1\leq i\leq 5$. Then, by Lemma~\ref{lem:maxEntropy},
\[\mathbb{H}(X_1,X_2,X_3,X_{4},X_5)=\mathbb{H}(f)=\log_2(\hom(C_5,G)).\] 

Label the vertices of $J$ by $1,\dots,10$ so that vertices $1,2,3,4$ and $5$ form a $5$-cycle, vertices $5,6,7$ and $8$ form a $4$-cycle and vertices $9$ and $10$ are in the two-vertex component of $J$. Our goal is to construct a $G$-homomorphism supported random variable $(Y_v: v\in V(J))$ which has high entropy. We start by applying Lemma~\ref{lem:glueEntropy} with $A_1=(X_1,\dots,X_4)$, $A_2=A_2'=X_5$ and $A_3=(X_1,X_2)$. This gives us random variables $(Z_1,\dots,Z_4)$, $Z_5$ and $(Z_6,Z_7)$ such that $(Z_1,\dots,Z_5)$ is distributed like $(X_1,\dots,X_5)$, $(Z_5,Z_6,Z_7)$ is distributed like $(X_5,X_1,X_2)$ and $(Z_1,\dots,Z_4)$ and $(Z_6,Z_7)$ are conditionally independent given $Z_5$. Applying this lemma again with $A_1=(Z_1,\dots,Z_4,Z_6)$, $A_2=(Z_5,Z_7)$, $A_2'=(X_5,X_2)$ and $A_3=X_1$ gives us random variables $Y_1,\dots,Y_8$ such that $(Y_1,\dots,Y_7)$ is distributed like $(Z_1,\dots,Z_7)$,  $(Y_5,Y_8,Y_7)$ is distributed like $(X_5,X_1,X_2)$ and $(Y_1,\dots,Y_4,Y_6)$ is conditionally independent of $Y_8$ given $(Y_5,Y_7)$. Finally, we simply let $(Y_9,Y_{10})$ be a copy of $(X_1,X_2)$ chosen independently of $(Y_1,\dots,Y_8)$. 

By Lemma~\ref{lem:chainRule}, $\mathbb{H}(Y_1,\dots,Y_{10})$ is equal to
\[\mathbb{H}(Y_1,\dots,Y_5) + \mathbb{H}(Y_6,Y_7\mid Y_1,\dots,Y_5) + \mathbb{H}(Y_8\mid Y_1,\dots,Y_7) + \mathbb{H}(Y_9,Y_{10}\mid Y_1,\dots,Y_8).\]
By Lemma~\ref{lem:deconditioning} and conditional independence, this is equal to
\[\mathbb{H}(Y_1,\dots,Y_5)+\mathbb{H}(Y_6,Y_7\mid Y_5) + \mathbb{H}(Y_8\mid Y_5,Y_7) + \mathbb{H}(Y_9,Y_{10}).\]
By construction, we may rewrite this expression as
\[\mathbb{H}(X_1,\dots,X_5)+\mathbb{H}(X_1,X_2\mid X_5) + \mathbb{H}(X_1\mid X_2,X_5) + \mathbb{H}(X_1,X_2).\]
Using Lemma~\ref{lem:deconditioning}, we can bound this below by 
\[\mathbb{H}(X_1,\dots,X_5)+\mathbb{H}(X_1,X_2\mid X_4,X_5) + \mathbb{H}(X_1\mid X_2,X_3,X_4,X_5) + \mathbb{H}(X_1,X_2).\]
By symmetry of $C_5$ and the fact that $f$ was chosen uniformly at random, $(X_1,X_2)$ and $(X_4,X_5)$ are identically distributed and so $\mathbb{H}(X_1,X_2)=\mathbb{H}(X_4,X_5)$. Likewise, $\mathbb{H}(X_1\mid X_2,X_3,X_4,X_5)=\mathbb{H}(X_3\mid X_1,X_2,X_4,X_5)$. So, the above expression can be rewritten as 
\[\mathbb{H}(X_1,\dots,X_5)+\mathbb{H}(X_1,X_2\mid X_4,X_5) + \mathbb{H}(X_3\mid X_1,X_2,X_4,X_5) + \mathbb{H}(X_4,X_5)\]
which, by Lemma~\ref{lem:chainRule}, is equal to $2\mathbb{H}(X_1,\dots,X_5)$. Thus, by Lemma~\ref{lem:maxEntropy},
\[\log_2(\hom(J,G))\geq \mathbb{H}(Y_1,\dots,Y_{10})\geq 2\mathbb{H}(X_1,\dots,X_5)=2\log_2(\hom(C_5,G))\]
which completes the proof.
\end{proof}

We are nearly ready to show that the first graph in Figure~\ref{fig:glueExamples} is common. The following standard lemma from graph limit theory will be used to convert inequalities for graphs into inequalities for graphons.

\begin{lem}[See~\cite{Lovasz12}]
\label{lem:limit}    
For $t\geq1$, let $H_1,\dots,H_t$ be graphs and $f_1,f_2:[0,1]^t\to \mathbb{R}$ be continuous functions. If every graph $G$ satisfies
\[f_1(t(H_1,G),\dots,t(H_t,G)) + f_2(t(H_1,\overline{G}),\dots,t(H_t,\overline{G})) \geq \epsilon(v(G))\]
where $\epsilon:\mathbb{N}\to \mathbb{R}$ is a function such that $\lim_{n\to\infty}\epsilon(n)=0$, then 
\[f_1(t(H_1,W),\dots,t(H_t,W)) + f_2(t(H_1,1-W),\dots,t(H_t,1-W))\geq0\]
for every graphon $W$.
\end{lem}

\begin{proof}[Proof Sketch]
Given a graphon $W$, a \emph{$W$-random graph} of order $n$ is the graph $G_{n,W}$ with vertex set $\{1,\dots,n\}$ obtained by sampling $n$ uniformly random points $x_1,\dots,x_n$ of $[0,1]$ independently and then, for $1\leq i\neq j\leq n$, adding an edge from $i$ to $j$ with probability $W(x_i,x_j)$. By standard concentration inequalities, with probability $1$, we have $\lim_{n\to\infty}t(H,G_{n,W})=t(H,W)$ and $\lim_{n\to\infty}t(H,\overline{G}_{n,W})=t(H,1-W)$ for every graph $H$. The result follows by continuity of $f_1$ and $f_2$. 
\end{proof}

\begin{prop}
\label{prop:H1common}
The first graph in Figure~\ref{fig:glueExamples} is common.
\end{prop}

\begin{proof}
Denote the first graph in Figure~\ref{fig:glueExamples} by $H$ and let $J:=H\sqcup K_2$. Our goal is to apply Corollary~\ref{cor:convexity} with $F=C_5$, $k=2$ and $\ell=1$. By Theorem~\ref{th:C5}, $C_5$ is strongly common. Note that \eqref{eq:H_ge_Fcor} and \eqref{eq:ehefcor} hold trivially. To verify \eqref{eq:thecorrelationcor}, we use Proposition~\ref{prop:H1} to get
\[t(K_2,G)t(H,G) = t(J,G)=\frac{\hom(J,G)}{n^{v(J)}} \geq \frac{\hom(C_5,G)^2}{n^{10}} = t(C_5,G)^2\]
for every graph $G$. By Lemma~\ref{lem:limit}, this implies that $t(K_2,W)t(H,W)\geq t(C_5,W)$ for every graphon $W$. Thus, \eqref{eq:thecorrelationcor} holds and so $H$ is common by Corollary~\ref{cor:convexity}. 
\end{proof}

\section{Gluing Templates and Generalized Trees}
\label{sec:gluing}

We open this section by describing a class of graphs obtained by gluing together subgraphs of a particular graph $F$ in a tree-like manner. For a set $X$, let $2^X$ denote the collection of all subsets of $X$. 

\begin{defn}
Let $F$ be a graph, $T$ be a tree and let $\psi:V(T)\cup E(T)\to 2^{V(F)}$. We say that $(T,\psi)$ is an \emph{$F$-gluing template} if $\psi(st)\subseteq \psi(s)\cap\psi(t)$ for every edge $st$ of $T$. 
\end{defn}

Next, we explain the way in which an $F$-gluing template gives rise to a graph. Given a graph $F$ and a subset $S\subseteq V(F)$, let $F[S]$ be the subgraph of $F$ induced by $S$.

\begin{defn}
Let $F$ be a graph and $(T,\psi)$ be an $F$-gluing template. The \emph{generalized $F$-tree} corresponding to $(T,\psi)$ is the graph $J(T,\psi)$ constructed in the following manner. Start by taking a copy $F_s$ of $F[\psi(s)]$ for each $s\in V(T)$. Then, for each $st\in E(T)$ and each $v\in \psi(st)$, we identify the vertex of $F_s$ corresponding to $v$ with the vertex of $F_t$ corresponding to $v$.
\end{defn}

Figure \ref{fig:F-tree-example} gives an example of a generalized $C_5$-tree. In this example, $J(T,\psi)$ is the first graph depicted in Figure \ref{fig:glueExamples}, which was shown to be common in Section \ref{sec:example}.

\begin{figure}[h]
    \centering
    \includegraphics{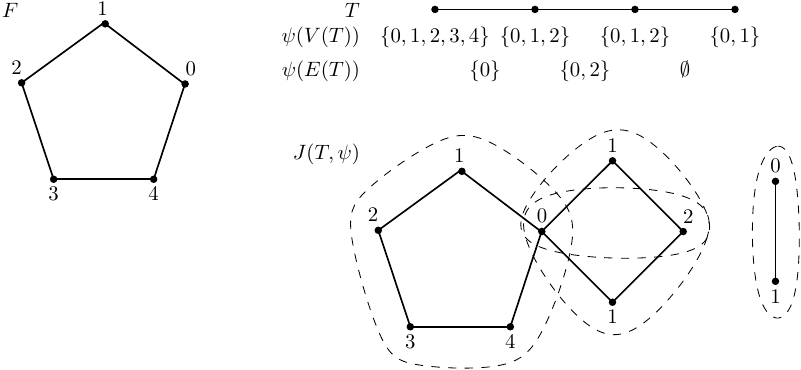}
    \caption{An example of a generalized $C_5$-tree}
    \label{fig:F-tree-example}
\end{figure}

\begin{rem}
\label{rem:v(J)}
Note that, for every $F$-gluing template $(T,\psi)$, the graph $J=J(T,\psi)$ satisfies
\[v(J) = \sum_{t\in V(T)}|\psi(t)| - \sum_{st\in E(T)}|\psi(st)|.\]
\end{rem}

We pause for a few basic examples. 

\begin{ex}
For every graph $F$ and tree $T$, if $\psi:V(T)\cup E(T)\to 2^{V(F)}$ maps each vertex of $T$ to $V(F)$ and each edge to $\emptyset$, then $J(T,\psi)$ is nothing more than the disjoint union of $v(T)$ copies of $F$.
\end{ex}

\begin{ex}
If $(T,\psi)$ is a $C_m$-gluing template such that $\psi(s)=V(C_m)$ for every $s\in V(T)$ and $\psi(st)$ is either a singleton or an edge of $C_m$ for all $st\in E(T)$, then $J(T,\psi)$ is a simple $C_m$-tree, as defined in the introduction. 
\end{ex}

The goal in the rest of this section is to describe a sufficient condition on an $F$-gluing template $(T,\psi)$ for the graph $J=J(T,\psi)$ to satisfy
\[t(J,G)\geq t(F,G)^{e(J)/e(F)}\]
for every graph $G$. As in the proof of Propositon~\ref{prop:H1}, our approach is to construct a distribution on the set $\Hom(J,G)$ of homomorphisms from $J$ to $G$ with entropy at least $e(J)/e(F)$ times the entropy of a uniformly random homomorphism from $F$ to $G$. The inequality will then follow from an application of Lemma~\ref{lem:maxEntropy}. This approach seems to have first appeared in the work of Kopparty and Rossman~\cite{KoppartyRossman11} and has been used in several papers related to Sidorenko and common graphs~\cite{BlekhermanRaymond22,ConlonKimLeeLee18,Behague+23+,Szegedy15,ConlonLee17,GrzesikLeeLidickyVolec22,Lee21}. The next lemma describes a way in which one can use a distribution on homomorphisms from $F$ to $G$ to get a distribution on homomorphisms from $J(T,\psi)$ to $G$ of high entropy.  

\begin{lem}
\label{lem:entropyTree}
Let $F$ and $G$ be graphs, let $(T,\psi)$ be an $F$-gluing template and let $J=J(T,\psi)$. For every $G$-homomorphism supported random variable $(X_v:v\in V(F))$ there exists $(Y_v: v\in V(J))$ which is $G$-homomorphism supported such that
\begin{equation}\label{eq:entropyTree}\mathbb{H}(Y_v: v\in V(J))=\sum_{s\in V(T)}\mathbb{H}(X_v: v\in \psi(s)) - \sum_{st\in E(T)}\mathbb{H}(X_v: v\in \psi(st)).\end{equation}
\end{lem}

\begin{proof}
For each $s\in V(T)$, let $V_s$ be the vertices of the copy $F_s$ of $F[\psi(s)]$ added during the construction of $J(T,\psi)$. Let $\gamma_s:V_s\to \psi(s)$ be the bijection which maps each vertex of $V_s$ to the corresponding vertex of $\psi(s)$. We prove, by induction on $v(T)$, that there exists a $G$-homomorphism supported random variable $(Y_v:v\in V(J))$ satisfying \eqref{eq:entropyTree} such that, additionally, for each $s\in V(T)$, we have that $(Y_v: v\in V_s)$ has the same distribution as $(X_{\gamma_s(v)}: v\in V_s)$. In the case that $v(T)=1$, simply set $(Y_v: v\in V_s)=(X_{\gamma_s(v)}: v\in V_s)$, where $s$ is the unique vertex of $T$.

Now, suppose that $v(T)\geq2$, let $\ell$ be a leaf of $T$ and let $w$ be the unique neighbour of $\ell$. Let $T'=T\setminus\{\ell\}$ and let $\psi'$ be the restriction of $\psi$ to $V(T')\cup E(T')$. Let $J'=J(T',\psi')$. By induction, we can construct a $G$-homomorphism supported $(Y_v: v\in V(J'))$ of entropy 
\[\sum_{s\in V(T')}\mathbb{H}(X_v: v\in \psi'(s)) - \sum_{st\in E(T')}\mathbb{H}(X_v: v\in \psi'(st))\]
such that $(Y_v: v\in V_s)$ has the same distribution as $(X_{\gamma_s(v)}: v\in V_s)$ for all $s\in V(T')$. To construct the desired random variable $(Y_v: v\in V(J))$, we apply Lemma~\ref{lem:glueEntropy}. Let $A_1=(Y_v: v\in V(J')\setminus(V_w\cap V_\ell))$, let $A_2=(Y_v: v\in V_w\cap V_\ell)$, let $A_2'=(X_{\gamma_w(v)}: v\in V_w\cap V_\ell)$ and let $A_3=(X_{\gamma_\ell(v)}: v\in V_\ell\setminus V_w)$. Note that $A_2$ and $A_2'$ are identically distributed by the inductive hypothesis. The existence of the desired distribution now follows from Lemma~\ref{lem:glueEntropy}.
\end{proof}

We would like to describe conditions under which the expression in \eqref{eq:entropyTree} can be bounded below by  $(e(J)/e(F))\mathbb{H}(X_v: v\in V(F))$. As in the proof of Proposition~\ref{prop:H1}, this will involve ``adding conditioning'' and doing some cancellation. The key is to use the following consequence of Lemmas~\ref{lem:chainRule} and~\ref{lem:deconditioning}.

\begin{lem}
\label{lem:replace}
Let $R_1,R_2$ and $R_3$ be finite disjoint sets. For every discrete random variable $(X_v:v\in R_1\cup R_2\cup R_3)$, 
\[\mathbb{H}(X_v: v\in R_1\cup R_2\cup R_3) - \mathbb{H}(X_v: v\in R_2\cup R_3)-\mathbb{H}(X_v: v\in R_1\cup R_2)+\mathbb{H}(X_v: v\in R_2)\]
is non-positive.
\end{lem}

\begin{proof}
By Lemma~\ref{lem:deconditioning},
\[\mathbb{H}(X_v:v\in R_1\mid X_v: v\in R_2\cup R_3)\leq \mathbb{H}(X_v: v\in R_1\mid X_v : v\in R_2).\]
By Lemma~\ref{lem:chainRule},
\[\mathbb{H}(X_v:v\in R_1\mid X_v: v\in R_2) = \mathbb{H}(X_v: v\in R_1\cup R_2) - \mathbb{H}(X_v:v\in R_2)\]
and
\[\mathbb{H}(X_v:v\in R_1\mid X_v: v\in R_2\cup R_3) = \mathbb{H}(X_v: v\in R_1\cup R_2\cup R_3) - \mathbb{H}(X_v:v\in R_2\cup R_3).\]
Substituting these equalities into the inequality above completes the proof. 
\end{proof}

\begin{rem}
Note that, since $\mathbb{H}(X_v: v\in \emptyset)=0$, applying Lemma~\ref{lem:replace} in the case $R_2=\emptyset$ yields
\[\mathbb{H}(X_v: v\in R_1\cup R_3) - \mathbb{H}(X_v: v\in R_3)-\mathbb{H}(X_v: v\in R_1)\leq0\]
for any two disjoint sets $R_1$ and $R_3$. 
\end{rem}

Next, we observe that, if $f$ is a uniformly random homomorphism from a graph $F$ to a graph $G$, then, for any two subsets $S_1$ and $S_2$ of $V(F)$ which are ``symmetric'' to one another, the entropy of $(f(v):v\in S_1)$ is the same as the entropy of $(f(v): v\in S_2)$; one may recall that this idea was used in the proof of Proposition~\ref{prop:H1}. To formalize this, recall that an \emph{automorphism} of $F$ is an isomorphism from $F$ to itself. Let $\Aut(F)$ denote the set of all automorphisms of $F$. Given a graph $F$, let $\sim_F$ be the equivalence relation on $2^{V(F)}$ such that $S_1\sim_FS_2$ if there exists $\varphi\in\Aut(F)$ such that $\varphi(S_1)=S_2$. Let $\mathcal{C}_F$ be a subset of $2^{V(F)}$ containing one element from each equivalence class of $\sim_F$.

Consider the vector space $\mathbb{R}^{\mathcal{C}_F\setminus\{\emptyset\}}$ spanned by vectors $\vec{e}_S$ for $S\in\mathcal{C}_F\setminus\{\emptyset\}$. Let $\vec{e}_\emptyset$ be the zero vector and, for $S'\subseteq V(F)$, let $\vec{e}_{S'}=\vec{e}_S$ where $S\in\mathcal{C}_F$ and $S'\sim_F S$. Given an $F$-gluing template $(T,\psi)$, define
\[\vec{z}_{T,\psi}:=\sum_{s\in V(T)}\vec{e}_{\psi(s)}-\sum_{st\in E(T)}\vec{e}_{\psi(st)}.\]
For any pairwise disjoint subsets $R_1,R_2$ and $R_3$ of $V(F)$, define
\[\vec{x}_{R_1,R_2,R_3} := \vec{e}_{R_1\cup R_2\cup R_3} - \vec{e}_{R_2\cup R_3}-\vec{e}_{R_1\cup R_2}+\vec{e}_{R_2}.\]
In particular, if $R_2=\emptyset$, then
\[\vec{x}_{R_1,\emptyset,R_3} := \vec{e}_{R_1\cup R_3} - \vec{e}_{R_3}-\vec{e}_{R_1}.\]
One should observe that the coefficients in the expression for $\vec{z}_{T,\psi}$ mirror those on the right side of \eqref{eq:entropyTree} while the coefficients in the expression for $\vec{x}_{R_1,R_2,R_3}$ mirror those in the expression in Lemmas~\ref{lem:replace}. We can now reduce the problem of using Lemma~\ref{lem:replace} to bound \eqref{eq:entropyTree} below by $(e(J)/e(F))\mathbb{H}(X_v: v\in V(F))$ to solving a specific linear program.

\begin{defn}
Let $F$ be a graph, $(T,\psi)$ be an $F$-gluing template and $J=J(T,\psi)$. We say that $(T,\psi)$ is \emph{good} if $e(J)>0$ and the vector
\[(e(J)/e(F))\vec{e}_{V(F)} - \vec{z}_{T,\psi}\]
is in the convex cone generated by the vectors $\vec{x}_{R_1,R_2,R_3}$ over all pairwise disjoint $R_1,R_2,R_3$.
\end{defn}

Next, we relate the number of edges and vertices of $J(T,\psi)$ to that of $F$ when $(T,\psi)$ is a good $F$-gluing template.

\begin{lem}
\label{lem:sameDensity}
If $(T,\psi)$ is a good $F$-gluing template and $J=J(T,\psi)$, then
\[e(J)/v(J) = e(F)/v(F).\]
\end{lem}

\begin{proof}
Given an $F$-gluing template $(T,\psi)$, define
\[\vec{\zeta}_{T,\psi}:=\sum_{s\in V(T)}|\psi(s)|\cdot\vec{e}_{\psi(s)}-\sum_{st\in E(T)}|\psi(st)|\cdot\vec{e}_{\psi(st)}.\]
For pairwise disjoint subsets $R_1,R_2$ and $R_3$ of $V(F)$, define
\[\vec{\xi}_{R_1,R_2,R_3} := |R_1\cup R_2\cup R_3|\cdot\vec{e}_{R_1\cup R_2\cup R_3} - |R_2\cup R_3|\cdot\vec{e}_{R_2\cup R_3}-|R_1\cup R_2|\cdot\vec{e}_{R_1\cup R_2}+|R_2|\cdot\vec{e}_{R_2}.\]
Note that, for every set $S\subseteq V(F)$, the coefficient of $\vec{e}_S$ in $\vec{\zeta}_{T,\psi}$ is equal to $|S|$ times the coefficient of $\vec{e}_S$ in $\vec{z}_{T,\psi}$. Likewise, the coefficient of $\vec{e}_S$ in  $\vec{\xi}_{R_1,R_2,R_3}$ is $|S|$ times the coefficient of $\vec{e}_S$ in $x_{R_1,R_2,R_3}$.

Now, since $(T,\psi)$ is good, we can write
\begin{equation}\label{eq:zthingy}\vec{z}_{T,\psi} + \sum_{R_1,R_2,R_3}c_{R_1,R_2,R_3}\cdot \vec{x}_{R_1,R_2,R_3}=(e(J)/e(F))\cdot \vec{e}_{V(F)}\end{equation}
for non-negative constants $c_{R_1,R_2,R_3}$. Now, consider the expression
\[\vec{\zeta}_{T,\psi} + \sum_{R_1,R_2,R_3}c_{R_1,R_2,R_3}\cdot \vec{\xi}_{R_1,R_2,R_3}.\]
As discussed in the previous paragraph, for each $S\subseteq V(F)$, the coefficient $\vec{e}_S$ in this expression is equal to $|S|$ times the coefficient of $\vec{e}_S$ in the expression on the left side of \eqref{eq:zthingy}. Therefore, by \eqref{eq:zthingy}, 
\begin{equation}\label{eq:zetathingy}\vec{\zeta}_{T,\psi} + \sum_{R_1,R_2,R_3}c_{R_1,R_2,R_3}\cdot \vec{\xi}_{R_1,R_2,R_3}=(e(J)/e(F))v(F)\cdot \vec{e}_{V(F)}.\end{equation}
Let $\vec{j}$ denote the all-ones vector in $\mathbb{R}^{\mathcal{C}_F}$ and $\langle\cdot,\cdot\rangle$ denote the standard inner product. By Remark~\ref{rem:v(J)}, 
\[v(J) = \left\langle\vec{\zeta}_{T,\psi},\vec{j}\right\rangle\]
which, by \eqref{eq:zetathingy}, is equal to
\[\left\langle(e(J)/e(F))v(F)\cdot \vec{e}_{V(F)} - \sum_{R_1,R_2,R_3}c_{R_1,R_2,R_3}\cdot \vec{\xi}_{R_1,R_2,R_3},\vec{j}\right\rangle\]
\[=(e(J)/e(F))v(F) - \sum_{R_1,R_2,R_3}c_{R_1,R_2,R_3}\cdot \left(|R_1\cup R_2\cup R_3|-|R_2\cup R_3|-|R_1\cup R_2|+|R_2|\right).\]
This final expression is simply equal to $(e(J)/e(F))v(F)$ since the summation is over pairwise disjoint sets $R_1,R_2,R_3$. Thus, $v(J)=(e(J)/e(F))v(F)$ and the proof is complete.
\end{proof}

The key lemma of this section is as follows.

\begin{lem}
\label{lem:goodGluing}
If $(T,\psi)$ is a good $F$-gluing template and $J=J(T,\psi)$, then
\[t(J,G)\geq t(F,G)^{e(J)/e(F)}\]
for every graph $G$.
\end{lem}

\begin{proof}
If there are no homomorphisms from $F$ to $G$, then $t(F,G)=0$ and the inequality holds trivially. Otherwise, $f$ be a uniformly random homomorphism from $F$ to $G$ and, for each $v\in V(F)$, let $X_v:=f(v)$. Note that, by the uniformity of $f$, we have that $\mathbb{H}(X_v: v\in S)=\mathbb{H}(X_v: v\in S')$ whenever $S\sim_FS'$; this will be applied many times in the rest of the proof. Let $(Y_v: v\in V(J))$ be a $G$-homomorphism supported random variable constructed from $(X_v: v\in V(F))$ as in Lemma~\ref{lem:entropyTree}. By \eqref{eq:entropyTree}, 
\begin{equation}\label{eq:entropyTreeAgain}\mathbb{H}(Y_v:v\in V(J))=\sum_{s\in V(T)}\mathbb{H}(X_v: v\in\psi(s)) - \sum_{st\in E(T)}\mathbb{H}(X_v: v\in \psi(st)).\end{equation}

Since $(T,\psi)$ is good, we can choose $c_{R_1,R_2,R_3}\geq0$ for each triple $R_1,R_2,R_3$ of pairwise disjoint subsets of $V(F)$ such that so that 
\begin{equation}\label{eq:vectorSpace}(e(J)/e(F))\vec{e}_{V(F)} = \sum_{s\in V(T)}\vec{e}_{\psi(s)}-\sum_{st\in E(T)}\vec{e}_{\psi(st)}+ \sum_{R_1,R_2,R_3}c_{R_1,R_2,R_3}\cdot \vec{x}_{R_1,R_2,R_3}.\end{equation}
By Lemma~\ref{lem:replace} and the fact that $c_{R_1,R_2,R_3}\geq0$ for all $R_1,R_2,R_3$, the linear combination consisting of the sum over all $R_1,R_2,R_3$ of $c_{R_1,R_2,R_3}$ times
\[\mathbb{H}(X_v: v\in R_1\cup R_2\cup R_3) - \mathbb{H}(X_v: v\in R_2\cup R_3)-\mathbb{H}(X_v: v\in R_1\cup R_2)+\mathbb{H}(X_v: v\in R_2)\]
is non-positive. By \eqref{eq:vectorSpace}, adding this linear combination to the right side of \eqref{eq:entropyTreeAgain} yields $(e(J)/e(F))\mathbb{H}(X_v: v\in V(F))$. So, 
\[\mathbb{H}(Y_v: v\in V(J))\geq (e(J)/e(F))\mathbb{H}(X_v: v\in V(F)).\]
Combining this with two applications of Lemma~\ref{lem:maxEntropy} and the fact that $f$ is uniform, we get
\[\log_2(\hom(J,G)) \geq \mathbb{H}(Y_v: v\in V(J))\]
\[\geq (e(J)/e(F))\mathbb{H}(X_v: v\in V(F)) = (e(J)/e(F))\log_2(\hom(F,G)).\]
Therefore,
\[t(J,G)=\frac{\hom(J,G)}{v(G)^{v(J)}} \geq \frac{\hom(F,G)^{e(J)/e(F)}}{v(G)^{v(J)}} = \left(\frac{\hom(F,G)}{v(G)^{v(J)(e(F)/e(J))}}\right)^{e(J)/e(F)}.\]
By Lemma~\ref{lem:sameDensity}, we have $v(J)(e(F)/e(J))=v(F)$ and so the right side is precisely equal to $t(F,G)^{e(J)/e(F)}$. This completes the proof.
\end{proof}

\begin{cor}
\label{cor:deleteK2s}
Let $F$ be a graph and $\ell\geq0$. If $(T,\psi)$ is a good $F$-gluing template and $H$ is obtained from $J(T,\psi)$ by deleting $\ell$ two-vertex components and an arbitrary number of one-vertex components, then
\[t(H,G)t(K_2,G)^\ell\geq t(F,G)^{(e(H)+\ell)/e(F)}\]
for every graph $G$.
\end{cor}

\begin{proof}
This follows from Lemma~\ref{lem:goodGluing} and the facts that $e(J)=e(H)+\ell$ and $t(H,G)t(K_2,G)^\ell=t(J,G)$ for every graph $G$.
\end{proof}

\section{Generalized Trees of Odd Cycles}
\label{sec:oddCycles}

In this section, we apply the results built up so far to show that certain pairs of generalized $C_m$-trees are $(p_1,p_2)$-common for some $p_1,p_2\in (0,1)$ such that $p_1+p_2=1$. The following is our most general result of this type, from which Theorems~\ref{th:Sample} and~\ref{th:SimpleTrees} will be derived.

\begin{thm}
\label{th:cycleTree}
Let $m\geq3$ be odd and, for $i\in\{1,2\}$, let $(T_i,\psi_i)$ be a good $C_m$-gluing template, let $J_i=J(T_i,\psi_i)$ and let $H_i$ be obtained from $J_i$ by deleting $\ell_i\geq0$ two-vertex components and an arbitrary number of one-vertex components. If $e(H_i)\geq m$ for $i\in \{1,2\}$ and $p_1,p_2\in(0,1)$ such that $p_1+p_2=1$ and
\[\frac{e(H_1)+\ell_1}{e(H_1)p_1^{m-1}} = \frac{e(H_2)+\ell_2}{e(H_2)p_2^{m-1}},\]
then $(H_1,H_2)$ is $(p_1,p_2)$-common. 
\end{thm}

\begin{proof}
For $i\in \{1,2\}$, define 
\[k_i:=\frac{e(H_i)+\ell_i}{m}=\frac{e(J_i)}{m}.\]
We verify the hypotheses of Lemma~\ref{lem:convexity} with $F=C_m$. The fact that $C_m$ is strongly common follows from Theorem~\ref{th:OddCycles}. The conditions \eqref{eq:H_ge_F} and \eqref{eq:p1p2} hold by the hypotheses of the theorem and \eqref{eq:ehef} holds by definition of $k_1$ and $k_2$. By Corollary~\ref{cor:deleteK2s}, \eqref{eq:thecorrelation} holds as well. Thus, the result follows by Lemma~\ref{lem:convexity}. 
\end{proof}

Our goal in the rest of this section is to derive Theorems~\ref{th:Sample} and~\ref{th:SimpleTrees} from Theorem~\ref{th:cycleTree}. We restate these results for convenience. 

\repeatThm{Sample}

\begin{proof}
The first graph in Figure~\ref{fig:glueExamples} was shown to be common in Proposition~\ref{prop:H1common}. So, we focus on the second and third graphs, which we denote by $H_2$ and $H_3$, respectively. We describe two good $C_5$-gluing templates $(T_2,\psi_2)$ and $(T_3,\psi_3)$ such that, for $i\in\{2,3\}$, the graph obtained from $J_i(T_i,\psi_i)$ by deleting all components on at most two vertices is precisely $H_i$. By Theorem~\ref{th:cycleTree}, this is sufficient for showing that $H_i$ is common for $i\in \{2,3\}$. 

First, let $T_2=P_6$ where the vertices of $P_6$ are labelled $1,2,\dots,6$ in the order that they come on the path. Let
\[\psi_2(1)=\psi_2(2)=\psi_2(4)=\{1,2,3,4,5\},\quad \psi_2(3)=\{1,2,3\},\quad \psi_2(5)=\{1,2\},\quad \psi_2(6)=\{1\}.\] 
Also, let 
\[\psi_2(12)=\{2,3,4,5\},\quad \psi_2(23)=\{1\}, \quad\psi_2(34)=\{3\},\quad \psi_2(45)=\psi_2(56)=\emptyset.\] 
We have that $J(T_2,\psi_2)=H_2\sqcup K_2\sqcup K_1$. To see that $(T_2,\psi_2)$ is good, note that
\[\vec{z}_{(T_2,\psi_2)} = 3\vec{e}_{\{1,2,3,4,5\}} + \vec{e}_{\{1,2,3\}}+\vec{e}_{\{1,2\}}+\vec{e}_{\{1\}} - \vec{e}_{\{1,2,3,4\}}-\vec{e}_{\{1\}}-\vec{e}_{\{3\}}\]
which, since $\vec{e}_{\{1\}}=\vec{e}_{\{3\}}$, is equal to
\[3\vec{e}_{\{1,2,3,4,5\}} + \vec{e}_{\{1,2,3\}}+\vec{e}_{\{1,2\}}- \vec{e}_{\{1,2,3,4\}}-\vec{e}_{\{1\}}.\]
Starting with $\vec{z}_{(T_2,\psi_2)}$ and adding the vectors
\[\vec{x}_{\{1\},\{2,3\},\{4\}} = \vec{e}_{\{1,2,3,4\}}-\vec{e}_{\{2,3,4\}}-\vec{e}_{\{1,2,3\}}+\vec{e}_{\{2,3\}}=\vec{e}_{\{1,2,3,4\}}-2\vec{e}_{\{1,2,3\}}+\vec{e}_{\{1,2\}}\]
and
\[\vec{x}_{\{1\},\{2\},\{3\}} = \vec{e}_{\{1,2,3\}}-\vec{e}_{\{2,3\}}-\vec{e}_{\{1,2\}}+\vec{e}_{\{2\}}=\vec{e}_{\{1,2,3\}}-2\vec{e}_{\{1,2\}}+\vec{e}_{\{1\}}\]
yields $3\vec{e}_{\{1,2,3,4,5\}}$. Therefore, $(T_2,\psi_2)$ is good. By Theorem~\ref{th:cycleTree}, we get that $H_2$ is common. 

Next, let $T_3=P_5$ where the vertices are labelled $1,\dots,5$ in the order that they come on the path. Let
\[\psi_3(1)=\psi_3(2)=\psi_3(3)=\{1,2,3,4,5\},\quad \psi_3(4)=\psi_3(5)=\{1,2\}.\] 
Also, let 
\[\psi_3(12)=\{5\},\quad \psi_3(23)=\{2,3,4\}, \quad\psi_3(34)=\psi_3(45)=\emptyset.\] 
We have that $J(T_3,\psi_3)=H_3\sqcup K_2\sqcup K_2$. To see that $(T_3,\psi_3)$ is good, note that
\[\vec{z}_{(T_3,\psi_3)} = 3\vec{e}_{\{1,2,3,4,5\}} +2\vec{e}_{\{1,2\}} - \vec{e}_{\{5\}} - \vec{e}_{\{2,3,4\}}.\]
Starting with $\vec{z}_{(T_3,\psi_3)}$ and adding the vector
\[\vec{x}_{\{1\},\{2\},\{3\}} = \vec{e}_{\{1,2,3\}}-\vec{e}_{\{2,3\}}-\vec{e}_{\{1,2\}}+\vec{e}_{\{2\}}=\vec{e}_{\{2,3,4\}}-2\vec{e}_{\{1,2\}}+\vec{e}_{\{5\}}\]
yields $3\vec{e}_{\{1,2,3,4,5\}}$. Therefore, $(T_3,\psi_3)$ is good and so $H_3$ is common by Theorem~\ref{th:cycleTree}.
\end{proof}

Next, we prove Theorem~\ref{th:SimpleTrees}. The ``only if'' part of the theorem uses the following result of~\cite{FirstPaper}. Given a graph $H$ and an integer $m\geq3$, let $c_m(H)$ be the number of unlabelled cycles of length $m$ in $H$. The \emph{girth} of $H$ is $g(H):=\min\{m: c_m(H)>0\}$. 

\begin{thm}[Behague, Morrison and Noel~\cite{FirstPaper}]
\label{th:girthThm}
Let $m\geq3$ be odd and suppose that $g(H_1)=g(H_2)=m$. If $(H_1,H_2)$ is $(p_1,p_2)$-common, then
\[\frac{c_m(H_1)}{e(H_1)p_1^{m-1}}=\frac{c_m(H_2)}{e(H_2)p_2^{m-1}}.\]
\end{thm}

\repeatThm{SimpleTrees}

\begin{proof}
Let $H_1$ and $H_2$ be simple $C_m$-trees. We label the vertices of $C_m$ by $1,\dots,m$ in the order that they appear on the cycle. By definition, a graph $H$ is a simple $C_m$-tree if and only if $H=J(T',\psi')$ for some $C_m$-gluing template $(T',\psi')$ such that $\psi(s)=V(C_m)$ for all $s\in V(T')$ and $\psi(st)$ is either a singleton or a set consisting of a pair of adjacent vertices of $C_m$ for all $st\in E(T')$. For $i\in\{1,2\}$, let $(T_i',\psi_i')$ be such a $C_m$-gluing template for $H_i$ and let $\ell_i$ be the number of $st\in E(T_i')$ such that $|\psi_i'(st)|=2$. Then, for each $i\in\{1,2\}$,
\[g(H_i)=m,\]
\[v(H_i)=m\cdot v(T_i') - e(T_i')-\ell_i,\]
\[e(H_i)=m\cdot v(T_i')-\ell_i\]
and
\[c_m(H_i) = v(T_i').\]
In particular,
\begin{equation}\label{eq:edgeVerteexCm}e(H_i)-v(H_i)+1 = e(T_i')+1=v(T_i')=c_m(H_i) = \frac{e(H_i)+\ell_i}{m}.\end{equation}
If $(H_1,H_2)$ is $(p_1,p_2)$-common, then, by Theorem~\ref{th:girthThm} and \eqref{eq:edgeVerteexCm},
\[\frac{e(H_1)-v(H_1)+1}{e(H_1)p_1^{m-1}} = \frac{c_m(H_1)}{e(H_1)p_1^{m-1}}=\frac{c_m(H_2)}{e(H_2)p_2^{m-1}}=\frac{e(H_2)-v(H_2)+1}{e(H_2)p_2^{m-1}}.\]
This proves the ``only if'' direction. 

To prove the ``if'' direction, let $(T_i,\psi_i)$ be a $C_m$-gluing template obtained from $(T_i',\psi_i')$ as follows. Let $T_i$ be a tree constructed from $T_i'$ by adding $e(T_i')$ vertices in an arbitrary manner. Let $\psi_i$ agree with $\psi_i'$ on all vertices and edges of $T_i'$ and let it map $\ell_i$ of the new vertices to $\{1,2\}$, map the rest of the new vertices to $\{1\}$ and map every edge incident to at least one of the new vertices to $\emptyset$. 
It is easy to see that $J_i:=J(T_i,\psi_i)$ is the disjoint union of $H_i$ with $\ell_i$ copies of $K_2$ and $e(T_i')-\ell_i$ copies of $K_1$. As $\vec{z}_{T_i,\psi_i} = (e(J_i)/m)\vec{e}_{V(C_m)}$, we see that $(T_i,\psi_i)$ is good. By Theorem~\ref{th:cycleTree}, we have that $(H_1,H_2)$ is $(p_1,p_2)$-common provided that 
\[\frac{e(H_1)+\ell_1}{me(H_1)p_1^{m-1}} = \frac{e(H_2)+\ell_2}{me(H_2)p_2^{m-1}}.\]
By \eqref{eq:edgeVerteexCm}, the above equality is equivalent to
\[\frac{e(H_1)-v(H_1)+1}{e(H_1)p_1^{m-1}} = \frac{e(H_2)-v(H_2)+1}{e(H_2)p_2^{m-1}}\]
and the proof is complete. 
\end{proof}

\section{An Uncommon Graph in a Common Pair}
\label{sec:supersat}

Next, we present the proof of Theorem~\ref{th:K3unionK2}, restated below for convenience. The proof is inspired by an approach from a recent paper on disconnected common graphs~\cite{LeeNoel23+}. We will apply the following classical ``supersaturation version'' of Mantel's Theorem on the edge density of triangle-free graphs. 

\begin{thm}[Goodman's Supersaturation Bound~\cite{Goodman59}]
\label{th:GoodmanSupersat}
For every graphon $W$,
\[t(K_3,W)\geq t(K_2,W)(2t(K_2,W)-1).\]
\end{thm}
Recall that $D$ denotes the graph obtained from $K_4$ by deleting one edge.
\repeatThm{K3unionK2}

\begin{proof}
Let $W$ be a graphon. Our goal is to show that
\begin{equation}\label{eq:thegoal}\frac{t(D,W)}{5p^4} + \frac{t(K_3\sqcup K_2,1-W)}{4(1-p)^{3}}\geq \frac{p}{5}+\frac{1-p}{4}.\end{equation}
If $t(K_2,W)=0$, then the left side of \eqref{eq:thegoal} is $\frac{1}{4(1-p)^3}$ which is greater than $\frac{p}{4}+\frac{1-p}{5}$. Likewise, if $t(K_2,W)=1$, then the right side of \eqref{eq:thegoal} is $\frac{1}{5p^4}$ which is also greater than $\frac{p}{4}+\frac{1-p}{5}$ for the given value of $p$. So, we can assume that 
\begin{equation}\label{eq:notExtreme}0<t(K_2,W) = 1-t(K_2,1-W)<1.\end{equation}

Define
\[y:=t(K_3,W)-t(K_2,W)^3.\]
Then, by definition of $y$,
\begin{equation}\label{eq:K3W}t(K_3,W)=t(K_2,W)^3+y.\end{equation} 
By Theorem~\ref{th:K3stronglycommon}, we have
\begin{equation}\label{eq:K31-W}t(K_3,1-W)\geq t(K_2,W)^3 + t(K_2,1-W)^3-t(K_3,W) =t(K_2,1-W)^3-y.\end{equation}
Also, by Theorem~\ref{th:GoodmanSupersat} and the definition of $y$, 
\begin{equation}\label{eq:yLB}y \geq t(K_2,W)(2t(K_2,W)-1)-t(K_2,W)^3.\end{equation}
Note that $D$ is a simple $K_3$-tree. So, by Lemma~\ref{lem:goodGluing} (as used in the proof of Theorem~\ref{th:SimpleTrees}) and \eqref{eq:K3W},
\[t(K_2,W)t(D,W)\geq t(K_3,W)^2 = (t(K_2,W)^3+y)^2\]
which implies that
\[t(D,W)\geq \frac{(t(K_2,W)^3+y)^2}{t(K_2,W)}.\]
Note that, by \eqref{eq:notExtreme}, the above inequality does not involve a division by zero. By \eqref{eq:K3W},
\[t(K_3\sqcup K_2,1-W)=t(K_3,1-W)t(K_2,1-W)\geq (t(K_2,1-W)^3-y)t(K_2,1-W).\]
Now, define $x:=t(K_2,W)-p$ so that $t(K_2,W)=p+x$ and $t(K_2,1-W)=1-p-x$. By \eqref{eq:notExtreme}, we have $-p<x<1-p$. By the lower bounds on $t(D,W)$ and $t(K_3\sqcup K_2,1-W)$ derived above, we can bound the left side of \eqref{eq:thegoal} below by the minimum of
\[f(x,y):=\frac{((p+x)^3+y)^2}{5p^4(p+x)} + \frac{(1-p-x)((1-p-x)^3-y)}{4(1-p)^3}\]
over all real numbers $x$ and $y$ such that $-p< x< 1-p$ and $y\geq y_0(x)$ where $y_0(x):=(p+x)(2(p+x)-1) - (p+x)^3$; the constraint $y\geq y_0(x)$ comes from \eqref{eq:yLB}.  

First, we analyze the partial derivatives of $f$ with respect to $y$. We have
\[\frac{\partial f}{\partial y}= \frac{2(p+x)^2}{5p^4} + \frac{2y}{5p^4(p+x)}-\frac{1-p-x}{4(1-p)^3}.\]
For each fixed $x$ in the range $-p<x<1-p$, define
\[y_1(x) := \frac{5(1-p-x)(p+x)p^4}{8(1-p)^3} - (p+x)^3\] 
and note that $\frac{\partial f}{\partial y}(x,y_1(x))=0$ for all $x$. Also, we have
\[\frac{\partial^2 f}{\partial y^2} = \frac{2}{5p^4(p+x)}\]
which is positive. Thus, for each fixed $x$ in the range $-p<x<1-p$, the minimum of $f(x,y)$ over $y\geq y_0(x)$ is attained at $y=y_1(x)$ if $y_1(x)\geq y_0(x)$ or at $y=y_0(x)$ otherwise. 

Our next goal is to bound the function $g_0(x):=f(x,y_0(x))$ from below for all $x$ such that $-p<x<1-p$. We have
\[g_0(x) = \frac{((p+x)^3+y_0(x))^2}{5p^4(p+x)} + \frac{(1-p-x)((1-p-x)^3-y_0(x))}{4(1-p)^3}\]
\[ = \frac{(p+x)^2(2(p+x)-1)^2}{5p^4(p+x)} + \frac{(1-p-x)((1-p-x)^3 + (p+x)^3-(p+x)(2(p+x)-1))}{4(1-p)^3}\]
\[=\left(\frac{1065+336\sqrt{10}}{400}\right)x^3+\left(\frac{379+118\sqrt{10}}{80}\right)x^2-\left(\frac{135+72\sqrt{10}}{320}\right)x+\frac{52-3\sqrt{10}}{160}.\]
Thus, the minimum of $g_0(x)$ over all $-p<x<1-p$ is approximately $0.23263$, attained at $x\approx 0.057472$. So, $g_0(x)>\frac{p}{5}+\frac{1-p}{4}$ for all relevant choices of $x$. 

The final step is to bound $g_1(x):=f(x,y_1(x))$ for all $x$ satisfying $-p<x<1-p$ and $y_1(x)\geq y_0(x)$. First,
\[y_1(x)-y_0(x) = \frac{5(1-p-x)(p+x)p^4}{8(1-p)^3} - (p+x)(2(p+x)-1)\]
\[=\frac{p+x}{8(1-p)^3}\left(5(1-p-x)p^4 - 8(1-p)^3(2(p+x)-1)\right).\]
The above expression is non-negative if and only if
\[x\leq\frac{-5 p^5 + 21 p^4 - 56 p^3 + 72 p^2 - 40 p + 8}{5 p^4 - 16 p^3 + 48 p^2 - 48 p + 16} = \frac{-425 + 140 \sqrt{10}}{246}<0.08.\]
Thus, it suffices to show that $g_1(x)\geq \frac{p}{5}+\frac{1-p}{4}$ for all $-0.6\leq x\leq 0.08$, where the lower bound on $x$ comes from the fact that $p<0.6$. We have
\[g_1(x) = \frac{((p+x)^3+y_1(x))^2}{5p^4(p+x)} + \frac{(1-p-x)((1-p-x)^3-y_1(x))}{4(1-p)^3}\]
\[=\frac{5(1-p-x)^2(p+x)p^4}{64(1-p)^6}\]
\[+ \frac{(1-p-x)(8(1-p)^3(1-p-x)^3+8(1-p)^3(p+x)^3 - 5(1-p-x)(p+x)p^4)}{32(1-p)^6}\]
\[=-\left(\frac{487 + 154 \sqrt{10}}{100}\right)x^3+\left(\frac{79 + 25 \sqrt{10}}{50}\right)x^2 + \frac{7 + 2\sqrt{10}}{60}.\]
Thus, $\frac{dg_1}{dx}(0) = 0$. Also, 
\[\frac{d^2g_1}{dx^2}(x)= \frac{79+25\sqrt{10}}{25} - \left(\frac{1461 + 462 \sqrt{10}}{50}\right)x\]
which is positive for all $x\leq \frac{-6 + 2 \sqrt{10}}{3}\approx 0.108$. Therefore, the minimum of $g_1(x)$ over all $x$ in the range $-0.6\leq x\leq 0.08$ is attained at $x=0$. We are now done by observing that $g_1(0)=\frac{7+2\sqrt{10}}{60}=\frac{p}{5}+\frac{1-p}{4}$.
\end{proof}

\section{Conclusion}
\label{sec:concl}

As we have seen in this paper, strongly common graphs are very useful for generating examples of $(p_1,p_2)$-common pairs of graphs. A natural problem is to classify such graphs. 

\begin{prob}
\label{prob:stronglyCommon}
Classify strongly common graphs. 
\end{prob}

Note that, by the results of~\cite{ChenMa23+,Versteegen23+,KimLee22+}, the strongly common graphs with odd girth are precisely the odd cycles. It would be interesting to know whether there exists a non-bipartite strongly common graph which is not an odd cycle and, more ambitiously, a strongly common graph of high chromatic number. Examples of the latter would strengthen the recent breakthrough result of Kr\'a\v{l}, Volec and Wei~\cite{KralVolecWei22+}.

\begin{ques}
Does there exist a non-bipartite strongly common graph $H$ such that $H$ is not isomorphic to an odd cycle?
\end{ques}

\begin{ques}
Do there exist strongly common graphs of arbitrary chromatic number?
\end{ques}

In particular, it already seems difficult to determine whether a $4$-chromatic strongly common graph can exist. Given that Chen and Ma~\cite{ChenMa23+} proved that strongly common graphs other than $K_3$ must be triangle-free, perhaps a natural family of graphs to consider are those which arise from the Mycielski construction. The question of whether such graphs are common was asked by Conlon, Fox and Sudakov~\cite[Section~2.6]{ConlonFoxSudakov15}; it would also be interesting to know if they are strongly common.

The recent results of~\cite{ChenMa23+,Versteegen23+} suggest that the class of strongly common graphs is rather small, which limits the applicability of results like Lemma~\ref{lem:convexity}. It would therefore be interesting to extend the approach in this paper to graphs which satisfy a weaker hypothesis than being strongly common. 

Theorem~\ref{th:K3unionK2} demonstrates that, if $(H_1,H_2)$ is $(p,1-p)$-common for some $p\in (0,1)$, then $H_1$ and $H_2$ need not both be common. However, we conjecture that at least one of them must be common. 

\begin{conj}
Let $H_1$ and $H_2$ be graphs. If there exists $p\in (0,1)$ such that $(H_1,H_2)$ is $(p,1-p)$-common, then at least one of $H_1$ or $H_2$ is common.
\end{conj}

\bibliographystyle{plain}
\bibliography{common}

\appendix

\section{Proof of the Convexity Lemma}
\label{app:convexity}

We present the proof of Lemma~\ref{lem:convexity} which we restate here for convenience.

\begin{replemma}{lem:convexity}
Let $F$ be strongly common and let $H_1$ and $H_2$ be non-empty graphs. If $k_1,k_2,\ell_1,\ell_2$ are non-negative integers, $p_1,p_2\in (0,1)$ such that $p_1+p_2=1$ and conditions \eqref{eq:H_ge_F}--\eqref{eq:thecorrelation} below are satisfied, then $(H_1,H_2)$ is $(p_1,p_2)$-common.
\begin{enumerate}
    \item[\eqref{eq:H_ge_F}] $e(H_i) \ge e(F)$  for $i\in\{1,2\}$,
    \item[\eqref{eq:ehef}] $e(H_i)=k_ie(F)-\ell_i$ for $i\in \{1,2\}$,
    \item[\eqref{eq:p1p2}] $\frac{k_1}{e(H_1)p_1^{e(F)-1}}= \frac{k_2}{e(H_2)p_2^{e(F)-1}}$,
   \item[\eqref{eq:thecorrelation}]$t(H_i,W)t(K_2,W)^{\ell_i}\geq t(F,W)^{k_i}$ for every graphon $W$ and $i\in\{1,2\}$.
\end{enumerate}
\end{replemma}

\begin{proof}[Proof of Lemma~\ref{lem:convexity}]
Let $W_1$ and $W_2$ be graphons such that $W_1+W_2=1$. First, suppose that $t(K_2,W_1)=1$. In this case, we have $t(H_1,W_1)=1=(p_1+p_2)^{e(H_1)}$ and $t(H_2,W_2)=0=(p_2-p_2)^{e(H_2)}$. We apply Bernoulli's Inequality (i.e. the fact that $(1+x)^m\geq 1+mx$ for $x\geq -1$ and $m\geq1$) to get
\[\frac{t(H_1,W_1)}{e(H_1)p_1^{e(H_1)-1}} + \frac{t(H_2,W_2)}{e(H_2)p_2^{e(H_2)-1}} = \frac{(p_1+p_2)^{e(H_1)}}{e(H_2)p_2^{e(H_2)-1}} + \frac{(p_2-p_2)^{e(H_2)}}{e(H_1)p_1^{e(H_1)-1}}\]
\[= \frac{p_1(1+p_2/p_1)^{e(H_1)}}{e(H_1)} + \frac{p_2(1-1)^{e(H_1)}}{e(H_2)}\geq \frac{p_1 + p_2e(H_1)}{e(H_1)}+\frac{p_2-p_2e(H_2)}{e(H_2)}=\frac{p_1}{e(H_1)}+\frac{p_2}{e(H_2)}.\]
The case that $t(K_2,W_1)=0$ is similar. 

So, from here forward, we may assume that
\begin{equation}\label{eq:notZero}0<t(K_2,W_1),t(K_2,W_2)<1.\end{equation}
Without loss of generality, 
\begin{equation}\label{eq:wlog}t(F,W_1)-t(K_2,W_1)^{e(F)}\leq t(F,W_2)-t(K_2,W_2)^{e(F)}.\end{equation} 
Define
\[y:=\max\{0,t(K_2,W_1)^{e(F)}-t(F,W_1)\}.\]
Then, by definition of $y$, we have
\begin{equation}
\label{eq:tFW1}
t(F,W_1)\geq t(K_2,W_1)^{e(F)}-y\geq0.
\end{equation}
Since $F$ is strongly common,
\[t(F,W_1)+t(F,W_2)\geq t(K_2,W_1)^{e(F)}+t(K_2,W_2)^{e(F)}\]
and so, subtracting $t(F,W_1)$ from both sides,
\[t(F,W_2)\geq t(K_2,W_1)^{e(F)}+t(K_2,W_2)^{e(F)} - t(F,W_1).\]
If $y=t(K_2,W_1)^{e(F)}-t(F,W_1)\geq0$, then the above inequality translates to $t(F,W_2)\geq t(K_2,W_2)^{e(F)} + y\geq0$. On the other hand, if $y=0$, then $t(F_1,W_1)\geq t(K_2,W_1)^{e(F)}$ which, by \eqref{eq:wlog}, implies that $t(F_2,W_2)\geq t(K_2,W_2)^{e(F)}=t(K_2,W_2)^{e(F)}+y$. In either case, 
\begin{equation}
\label{eq:tFW2}
t(F_2,W_2)\geq t(K_2,W_2)^{e(F)}+y \geq0.
\end{equation}
Let $x:=t(K_2,W_1)-p_1$ so that $t(K_2,W_1)=p_1+x$ and $t(K_2,W_2)=p_2-x$. By \eqref{eq:notZero}, we have $-p_1<x<p_2$. By \eqref{eq:thecorrelation},
\[\frac{t(H_1,W_1)}{e(H_1)p_1^{e(H_1)-1}}+\frac{t(H_2,W_2)}{e(H_2)p_2^{e(H_2)-1}}\geq \frac{t(F,W_1)^{k_1}}{t(K_2,W_1)^{\ell_1}e(H_1)p_1^{e(H_1)-1}}+\frac{t(F,W_2)^{k_2}}{t(K_2,W_2)^{\ell_2}e(H_2)p_2^{e(H_2)-1}}\]
which, by \eqref{eq:tFW1} and \eqref{eq:tFW2}, is at least
\begin{equation}\label{eq:mainThingy} 
\frac{((p_1+x)^{e(F)}-y)^{k_1}}{(p_1+x)^{\ell_1}e(H_1)p_1^{e(H_1)-1}}+\frac{((p_2-x)^{e(F)}+y)^{k_2}}{(p_2-x)^{\ell_2}e(H_2)p_2^{e(H_2)-1}}.
\end{equation}
Define $f:(-p_1,p_2)\times \mathbb{R}\to \mathbb{R}$ by
\[f(x,y):=\frac{((p_1+x)^{e(F)}-y)^{k_1}}{(p_1+x)^{\ell_1}e(H_1)p_1^{e(H_1)-1}}+\frac{((p_2-x)^{e(F)}+y)^{k_2}}{(p_2-x)^{\ell_2}e(H_2)p_2^{e(H_2)-1}}\]
We will be done if we can show that $f(x,y)\geq\frac{p_1}{e(H_1)}+\frac{p_2}{e(H_2)}$ for all $x$ and $y$ in its domain such that $-(p_2-x)^{e(F)}\leq y\leq (p_1+x)^{e(F)}$; note that these constraints on $y$ come from the positivity constraints in \eqref{eq:tFW1} and \eqref{eq:tFW2}.

The partial derivative of $f$ with respect to $y$ is 
\begin{equation}\label{eq:partialy}\frac{\partial f}{\partial y}(x,y) = \frac{-k_1((p_1+x)^{e(F)}-y)^{k_1-1}}{(p_1+x)^{\ell_1}e(H_1)p_1^{e(H_1)-1}}+\frac{k_2((p_2-x)^{e(F)}+y)^{k_2-1}}{(p_2-x)^{\ell_2}e(H_2)p_2^{e(H_2)-1}}\end{equation}
\[=-\frac{k_1}{e(H_1)p_1^{e(F)-1}}\left(\frac{((p_1+x)^{e(F)}-y)^{k_1-1}}{(p_1+x)^{\ell_1}p_1^{e(H_1)-e(F)}}\right)+\frac{k_2}{e(H_2)p_2^{e(F)-1}}\left(\frac{((p_2-x)^{e(F)}+y)^{k_2-1}}{(p_2-x)^{\ell_2}p_2^{e(H_2)-e(F)}}\right)\]
By \eqref{eq:p1p2}, this can be simplified to 
\begin{equation}\label{eq:simplifiedPartial}\frac{\partial f}{\partial y}(x,y)  = \frac{k_1}{e(H_1)p_1^{e(F)-1}}\left(- \frac{((p_1+x)^{e(F)}-y)^{k_1-1}}{(p_1+x)^{\ell_1}p_1^{e(H_1)-e(F)}}+\frac{((p_2-x)^{e(F)}+y)^{k_2-1}}{(p_2-x)^{\ell_2}p_2^{e(H_2)-e(F)}}\right).\end{equation}
For each $x$ in the range $-p_1<x<p_2$, we have that $\frac{\partial f}{\partial y}(x,y)$ is negative at $y=-(p_2-x)^{e(F)}$ if $k_2 \ne 1$ and positive at $y=(p_1+x)^{e(F)}$ if $k_1 \ne 1$. Also, 
\[\frac{\partial^2 f}{\partial y^2}(x,y) = \frac{k_1}{e(H_1)p_1^{e(F)-1}}\left( \frac{(k_1-1)((p_1+x)^{e(F)}-y)^{k_1-2}}{(p_1+x)^{\ell_1}p_1^{e(H_1)-e(F)}}+\frac{(k_2-1)((p_2-x)^{e(F)}+y)^{k_2-2}}{(p_2-x)^{\ell_2}p_2^{e(H_2)-e(F)}}\right)\]
which is non-negative for all $x$ in the range $-p_1<x<p_2$ and $y$ in the range $-(p_2-x)^{e(F)}\leq y\leq (p_1+x)^{e(F)}$. Therefore, if $k_1,k_2 \ne 1$, for each $-p_1<x<p_2$, there is a unique choice of $y(x)$ such that $-(p_2-x)^{e(F)}<y(x)<(p_1+x)^{e(F)}$ and $\frac{\partial f}{\partial y}(x,y(x))=0$. 

If $k_1=1$ and $k_2 \ne 1$ then by \eqref{eq:ehef} we have $\ell_1 = 0$ and so 
\begin{equation}\frac{\partial f}{\partial y}(x,y)  = \frac{k_1}{e(H_1)p_1^{e(F)-1}}\left(
- 1 +
\frac{((p_2-x)^{e(F)}+y)^{k_2-1}}{(p_2-x)^{\ell_2}p_2^{e(H_2)-e(F)}}\right).\end{equation}
We observe that there is a unique choice of $y(x)$ such that $\frac{\partial f}{\partial y}(x,y) = 0$, but this $y(x)$ may not lie in the domain  $-(p_2-x)^{e(F)}\leq y\leq (p_1+x)^{e(F)}$. In this case, we simply extend the domain for $y$ to include this value and find a minimum for $f(x,y)$ on this larger domain. 
The case where $k_2 = 1$ and $k_1 \ne 1$ is similar. If $k_1 = k_2 = 1$, observe that $f(x,y)$ has no dependence on $y$ so we take $y(x) = 0$ for all $x$.

We pause to prove a claim. 

\begin{claim}
\label{claim:sameSign} 
For $-p_1< x< p_2$, if $x > 0$ then $y(x) \ge 0$, and if $x < 0$ then $y(x) \le 0$. 
\end{claim}

\begin{proof}[Proof of Claim]
Suppose that $0 < y(x)\leq(p_1+x)^{e(F)}$. Then, by definition of $y(x)$ and \eqref{eq:simplifiedPartial},
\[0=\frac{\partial f}{\partial y}(x,y(x)) = \frac{k_1}{e(H_1)p_1^{e(F)-1}}\left(- \frac{((p_1+x)^{e(F)}-y(x))^{k_1-1}}{(p_1+x)^{\ell_1}p_1^{e(H_1)-e(F)}}+\frac{((p_2-x)^{e(F)}+y(x))^{k_2-1}}{(p_2-x)^{\ell_2}p_2^{e(H_2)-e(F)}}\right)\]
\[> \frac{k_1}{e(H_1)p_1^{e(F)-1}}\left(- \frac{(p_1+x)^{e(F)(k_1-1)}}{(p_1+x)^{\ell_1}p_1^{e(H_1)-e(F)}}+\frac{(p_2-x)^{e(F)(k_2-1)}}{(p_2-x)^{\ell_2}p_2^{e(H_2)-e(F)}}\right)\]
where the strict inequality follows as $y(x) > 0$ implies $k_1$, $k_2$ are not both $1$.
By \eqref{eq:ehef} this is equal to
\[\frac{k_1}{e(H_1)p_1^{e(F)-1}}\left(-\left(\frac{p_1+x}{p_1}\right)^{e(H_1)-e(F)} + \left(\frac{p_2-x}{p_2}\right)^{e(H_2)-e(F)}\right).\]
If $x<0$, then this expression is $\ge 0$, which is a contradiction. Thus, if $y > 0$ then $x\geq0$ and, taking the contrapositive, if $x < 0$ then $y \le 0$. The proof that if $x > 0$ then $y\leq 0$ is similar.
\end{proof}

Now, for each $x\in (-p_1,p_2)$, let $g(x):=f(x,y(x))$. Note that 
\[g(0)=f(0,0) = \frac{p_1^{k_1e(F)}}{p_1^{e(H_1)-1+\ell_1}e(H_1)}+\frac{p_2^{k_2e(F)}}{p_2^{e(H_2)-1+\ell_2}e(H_2)}=\frac{p_1}{e(H_1)}+\frac{p_2}{e(H_2)}\]
where the last equality uses \eqref{eq:ehef}. So, we will be done if we can show that $g(x)$ is minimized at $x=0$. We have that $\frac{dg}{dx}$ is equal to
\[\frac{k_1((p_1+x)^{e(F)}-y(x))^{k_1-1}(e(F)(p_1+x)^{e(F)-1} -y'(x))}{(p_1+x)^{\ell_1}e(H_1)p_1^{e(H_1)-1}} - \frac{\ell_1((p_1+x)^{e(F)}-y(x))^{k_1}}{(p_1+x)^{\ell_1+1}e(H_1)p_1^{e(H_1)-1}}\]
\[+\frac{k_2((p_2-x)^{e(F)}+y(x))^{k_2-1}(-e(F)(p_2-x)^{e(F)-1} +y'(x))}{(p_2-x)^{\ell_2}e(H_2)p_2^{e(H_2)-1}} +\frac{\ell_2((p_2-x)^{e(F)}+y(x))^{k_2}}{(p_2-x)^{\ell_2+1}e(H_2)p_2^{e(H_2)-1}}\]
Recalling the expression for $\frac{\partial f}{\partial y}$ in \eqref{eq:partialy} and using the fact that $\frac{\partial f}{\partial y}(x,y(x))=0$, we see that all of the terms involving $y'(x)$ cancel, leaving us with
\[\frac{k_1((p_1+x)^{e(F)}-y(x))^{k_1-1}e(F)(p_1+x)^{e(F)-1}}{(p_1+x)^{\ell_1}e(H_1)p_1^{e(H_1)-1}} - \frac{\ell_1((p_1+x)^{e(F)}-y(x))^{k_1}}{(p_1+x)^{\ell_1+1}e(H_1)p_1^{e(H_1)-1}}\]
\[-\frac{k_2((p_2-x)^{e(F)}+y(x))^{k_2-1}e(F)(p_2-x)^{e(F)-1} }{(p_2-x)^{\ell_2}e(H_2)p_2^{e(H_2)-1}} +\frac{\ell_2((p_2-x)^{e(F)}+y(x))^{k_2}}{(p_2-x)^{\ell_2+1}e(H_2)p_2^{e(H_2)-1}}\]
\[=\left(\frac{((p_1+x)^{e(F)}-y(x))^{k_1-1}}{(p_1+x)^{\ell_1}p_1^{e(H_1)-e(F)}}\right)\left(\frac{k_1e(F)(p_1+x)^{e(F)}}{(p_1+x)e(H_1)p_1^{e(F)-1}} - \frac{\ell_1((p_1+x)^{e(F)}-y(x))}{(p_1+x)e(H_1)p_1^{e(F)-1}}\right)\]
\[+\left(\frac{((p_2-x)^{e(F)}+y(x))^{k_2-1}}{(p_2-x)^{\ell_2}p_2^{e(H_2)-e(F)}}\right)\left(-\frac{k_2e(F)(p_2-x)^{e(F)} }{(p_2-x)e(H_2)p_2^{e(F)-1}} +\frac{\ell_2((p_2-x)^{e(F)}+y(x))}{(p_2-x)e(H_2)p_2^{e(F)-1}}\right).\]
Using the expression for $\frac{\partial f}{\partial y}$ in \eqref{eq:simplifiedPartial} and the fact that $\frac{\partial f}{\partial y}(x,y(x))=0$, we observe that
\[\frac{((p_1+x)^{e(F)}-y(x))^{k_1-1}}{(p_1+x)^{\ell_1}p_1^{e(H_1)-e(F)}} = \frac{((p_2-x)^{e(F)}+y(x))^{k_2-1}}{(p_2-x)^{\ell_2}p_2^{e(H_2)-e(F)}}.\]
So, letting $b(x)$ denote this quantity and using \eqref{eq:ehef}, we can rewrite the above expression for $\frac{dg}{dx}$ as
\[=b(x)\left(\frac{k_1e(F)(p_1+x)^{e(F)}}{(p_1+x)e(H_1)p_1^{e(F)-1}} - \frac{\ell_1((p_1+x)^{e(F)}-y(x))}{(p_1+x)e(H_1)p_1^{e(F)-1}}\right)\]
\[+b(x)\left(-\frac{k_2e(F)(p_2-x)^{e(F)} }{(p_2-x)e(H_2)p_2^{e(F)-1}} +\frac{\ell_2((p_2-x)^{e(F)}+y(x))}{(p_2-x)e(H_2)p_2^{e(F)-1}}\right)\]
which, by \eqref{eq:ehef}, is equal to
\[b(x)\left(\left(\frac{p_1+x}{p_1}\right)^{e(F)-1} - \left(\frac{p_2-x}{p_2}\right)^{e(F)-1}\right)\]
\[+ b(x)y(x)\left(\frac{\ell_1}{(p_1+x)e(H_1)p_1^{e(F)-1}} + \frac{\ell_2}{(p_2-x)e(H_2)p_2^{e(F)-1}}\right).\]
Note that $b(x)>0$ for all $-p_1<x<p_2$. Thus, by Claim~\ref{claim:sameSign}, if $x>0$, then the first term in the above expression is positive and the second is non-negative and, if $x<0$, then the first term is negative and the second is non-positive. Therefore, the only critical point of $g$ is at $x=0$ and $g$ is decreasing to the left of this critical point and increasing to the right of it. This implies that $x=0$ is the global minimum of $g$ and completes the proof. 
\end{proof}

\end{document}